\documentclass[reqno]{amsart}
\usepackage[backend=bibtex, sorting=none, bibstyle=alphabetic, citestyle=alphabetic, sorting=nyt,maxbibnames=99, giveninits=true, isbn=false, url=false, doi=false, eprint=false]{biblatex}  
\renewbibmacro{in:}{}
\bibliography{references.bib}
\usepackage{amssymb}
\usepackage{calrsfs}
\usepackage{graphics,graphicx}
\usepackage{enumerate}
\usepackage{enumitem}
\usepackage{url}
\usepackage{xcolor}
\usepackage[ruled, lined, linesnumbered, longend]{algorithm2e}
\usepackage{hyperref}
\usepackage[justification=centering]{caption}
\usepackage{comment}
\usepackage{longtable}

\newtheorem{theorem}{Theorem}[section]
\newtheorem{lemma}[theorem]{Lemma}

\numberwithin{equation}{section}

\theoremstyle{remark}

\title{An explicit sub-Weyl bound for $\zeta(1/2 + it)$}
\author{Dhir Patel and Andrew Yang}
\date\today
\keywords{Van der Corput estimate, exponential sums, sub-Weyl bound, Riemann zeta function.}
\subjclass[2010]{Primary 11L07, 11M06}

\begin{document}

\begin{abstract}
In this article we prove an explicit sub-Weyl bound for the Riemann zeta function $\zeta(s)$ on the critical line $s = 1/2 + it$. In particular, we show that $|\zeta(1/2 + it)| \le 66.7\, t^{27/164}$ for $t \ge 3$. Combined, our results form the sharpest known bounds on $\zeta(1/2 + it)$ for $t \ge \exp(61)$.
\end{abstract}

\maketitle

\section{Introduction}
An important open problem in analytic number theory is the growth rate of the Riemann zeta-function $\zeta(s)$ on the critical line $s = 1/2 + it$ as $t \to \infty$. The well-known Lindel\"of Hypothesis asserts that $\zeta(1/2 + it) \ll_{\epsilon} t^{\epsilon}$ for any $\epsilon > 0$. Among the consequences of the hypothesis are many profound results for prime number distributions. Although the Lindel\"of Hypothesis is currently unproven, much effort have been expended to bound the zeta-function on the critical line, culminating in the current best-known bound of $\zeta(1/2 + it) \ll_{\epsilon} t^{13/84 + \epsilon}$ for any $\epsilon > 0$, due to Bourgain \cite{bourgain_decoupling_2016}.

In this article we are concerned with explicit bounds on $\zeta(1/2 + it)$. Such explicit bounds have been used to derive zero-free regions \cite{ford_zero_2002, mossinghoff_explicit_2022, yang_explicit_2023}, zero-density estimates \cite{kadiri_explicit_2018} and bounds on the argument of $\zeta(s)$ on the critical line \cite{trudgian_improved_2014, hasanalizade_counting_2021}. Recently, these results have in turn been used to obtain explicit theorems about prime distributions \cite{kadiri_short_2014, cullyhugill_primes_2021, broadbent_sharper_21, cullyhugill_explicit_2022, cullyhugill_on_2022, johnston_some_2022, fiori_density_2022, fiori_sharper_2022}, so there is substantial motivation to sharpen such bounds as much as possible. Nevertheless, known explicit bounds on $\zeta(1/2 + it)$ currently lag far behind the asymptotically sharpest-known results. Only two types of explicit subconvexity results are known --- the first being the classical van der Corput estimate of the form $|\zeta(1/2 + it)| \le At^{1/6}\log t$ for $t \ge t_0$ for some absolute constants $A$ and $t_0$. Such bounds are sometimes known as Weyl estimates because the exponent of $1/6$ was first achieved via the Weyl--Littlewood--Hardy method. The sharpest estimate of this type is due to \cite{hiary_improved_2022}, who built on the work of \cite{cheng_explicit_2004, trudgian_explicit_2015, hiary_explicit_2016} to prove
\begin{equation}\label{hpy_estimate}
|\zeta(1/2 + it)| \le 0.618 t^{1/6}\log t,\qquad t \ge 3. 
\end{equation}
A second type of explicit bound, known sometimes as sub-Weyl estimates, was first made explicit by Patel \cite{patel_explicit_2021}, who showed 
\begin{equation}\label{patel_subweyl}
|\zeta(1/2 + it)| \le 307.098 t^{27/164},\qquad t \ge 3. 
\end{equation}
Note that $27/164 = 0.164\ldots < 1/6$. In particular, \eqref{patel_subweyl} is the best-known explicit bound for the zeta-function on the critical line for $t \ge \exp(281)$. 

In this work we improve \eqref{patel_subweyl}. Our main result is 
\begin{theorem}\label{main_thm}
For $t \ge 3$, we have 
\[
|\zeta(1/2 + it)| \le 66.7 \, t^{27/164}. 
\]
\end{theorem}
Theorem \ref{main_thm} represents the sharpest known bound on $\zeta(1/2 + it)$ for $t \ge \exp(105)$. In \S \ref{sec:computation} we show that still sharper bounds are possible for smaller $t$. Together, our results form the best known bound for $t \ge \exp(61)$. Therefore, \eqref{hpy_estimate} remains sharper at $t \approx 3\cdot 10^{12}$, the verification height of the Riemann Hypothesis \cite{platt_riemann_2021}. This is significant since bounds for $\zeta(1/2 + it)$ near such values of $t$ are used in multiple explicit results \cite{hasanalizade_counting_2021, kadiri_explicit_2018, ford_zero_2002}.  

On the other hand, sharp bounds on $\zeta(1/2 + it)$ for larger values of $t$ are useful for deriving explicit zero-free regions \cite{ford_zero_2002, mossinghoff_explicit_2022, yang_explicit_2023}, for improved bounds on $S(T)$ \cite{trudgian_improved_2014, hasanalizade_counting_2021}, for refinements to Turing's method \cite{trudgian_improvements_2011, trudgian_improvements_2016}, and for asymptotically improved zero-density estimates \cite[Thm.\ 9.18]{titchmarsh_theory_1986}. For instance, following the method of \cite{kadiri_explicit_2018}, we may use Theorem \ref{main_thm} to prove an explicit zero-density result of the form 
\[
N(\sigma, T) \ll T^{\frac{109}{41}(1 - \sigma)}(\log T)^{4 - 3\sigma},
\]
where $N(\sigma, T)$ is the number of zeroes $\rho = \beta + i\gamma$ of $\zeta(s)$ with $\sigma < \beta < 1$ and $0 < \gamma < T$. 

\subsection{Approach and exponential sums}

As with all existing explicit bounds on $\zeta(1/2 + it)$, Theorem \ref{main_thm} relies on upper bounds on particular types of exponential sums, obtained via van der Corput's method of exponent pairs (for an exposition, see \cite{graham_kolesnik_1991}). Roughly stated, let $f(x)$ be a suitably well-defined and sufficiently smooth function satisfying $f'(x) \approx yx^{-\sigma}$ for some $y, \sigma > 0$. If $e(x) := \exp(2\pi i x)$, $0 \le k \le 1/2 \le l \le 1$ and
\[
S_f(a, N) := \sum_{a < n \le a + N}e(f(n)) \ll \left(\frac{y}{N^\sigma}\right)^{k}N^l,\qquad 0 < N \le a,
\]
then $(k, l)$ is an exponent pair. For instance, from the trivial bound $S_f(a, N) \ll N$ we see that $(0, 1)$ is an exponent pair. The motivation for studying exponent pairs is highlighted by the result that if $(k, l)$ is an exponent pair with $k + 2l \ge 3/2$, then
\[
\zeta(1/2 + it) \ll t^{(2k + 2l - 1)/4}\log t,
\]
see e.g.\ Phillips \cite{phillips_zetafunction_1933}.

The van der Corput method estimates $S_f(a, N)$ by iteratively transforming it into simpler exponential sums, via two processes. The $A$ process, also known as Weyl-differencing, expresses $S_f(a, N)$ in terms of $S_g(a, N)$, where $g(x)$ is a function of lower order than $f(x)$ (and is hence easier to control). By applying the $A$ process, we obtain that if $(k, l)$ is an exponent pair, then so is 
\[
A(k, l) := \left(\frac{k}{2k + 2}, \frac{k + l + 1}{2k + 2}\right).
\]
The $B$ process, also known as Poisson summation, expresses $S_f(a, b)$ in terms of another exponential sum that is typically shorter. Using the $B$ process, if $(k, l)$ is an exponent pair, then so is
\[
B(k, l) := \left(l - \frac{1}{2}, k + \frac{1}{2}\right).
\]

Favourable exponential pairs and, by extension, good estimates of $\zeta(1/2 + it)$, can be obtained by beginning with the trivial $(0, 1)$ exponent pair, then chaining together multiple applications of the $A$ and $B$ processes. The simplest van der Corput bound, such as \eqref{hpy_estimate}, is obtained from the exponent pair $AB(0, 1) = (1/6, 2/3)$. On the other hand, bounds such as \eqref{patel_subweyl} and Theorem \ref{main_thm} can be obtained using $ABA^3B(0, 1) = (11/82, 57/82)$. 

\subsection{Explicit exponent pairs}
Both the $A$ and $B$ processes have been made explicit. For the $A$ process, we have the following lemma, due to \cite{yang_explicit_2023} which builds on the work of \cite{cheng_explicit_2004, platt_improved_2015}.

\begin{lemma}[\cite{yang_explicit_2023} Lem.\ 2.3]\label{weyl_differencing}
Let $f(x)$ be real-valued and defined on $(a, a + N]$, for some integers $a, N$. For all integers $q > 0$, we have 
\[
(S_f(a, N))^2 \le \left(N - 1 + q\right)\left(\frac{N}{q} + \frac{2}{q}\sum_{r = 1}^{q - 1}\left(1 - \frac{r}{q}\right)S_{g_r}(a, N - r)\right)
\]
where $g_r(x) := f(x + r) - f(x)$. 
\end{lemma}

A general explicit version of the $B$ process was proved\footnote{We note here that in this general explicit version of $B$ process derived by Karatsuba and Korolev, one of the ``lower" order term, $K_2$, in their assertion could grow larger than the main-term given by the sum, $\sum c(n)Z(n)$ if $f''$ is small.} in Karatsuba and Korolev \cite{karatsuba_theorem_2007}, which relied on controlling the first four derivatives of the phase function $f(x)$. Patel \cite[Thm.\ 2.31]{patel_explicit_2021} proved the following explicit Poisson summation formula, which only relied on the first three derivatives.

\begin{lemma}[\cite{patel_explicit_2021} Thm. 2.31]\label{poisson_summation_formula}
Let $f(x)$ be three times differentiable. Let $f'(x)$ be decreasing in $[a, b]$ and $f'(b) = \alpha$, $f'(a) = \beta$. Further, let $x_\nu$ be defined by 
\[
f'(x_\nu) = \nu,\qquad \alpha < \nu \le \beta
\]
Furthermore suppose that $\lambda_2 \le |f''(x)| \le h_2\lambda_2$ and $|f'''(x)| \le h_3\lambda_3$. Then 
\begin{align*}
&\left|\sum_{a < n \le b}e(f(n)) - \sum_{\alpha < \nu \le \beta}\frac{e(f(x_{\nu}) - \nu x_\nu - 1/8)}{|f''(x_\nu)|^{1/2}}\right| \le \frac{40}{\sqrt{\pi}}\lambda_2^{-1/2}\\
&\qquad\qquad\qquad\qquad\qquad + A_1\log(\beta - \alpha + 4) + A_2(b - a)\lambda_2^{1/5}\lambda_3^{1/5} + A_3.
\end{align*}
where 
\[
A_1 = \frac{3 + 2h_2}{\pi},\quad A_2 = \frac{8}{(6\pi^3)^{1/5}} h_2h_3^{1/5}, \quad A_3 = \frac{1}{\pi}\left(4\gamma + \log 2 + \pi + \frac{20}{7}\right),
\]
and $\gamma = 0.577\ldots$ is the Euler-Mascheroni constant.  
\end{lemma}

In practical application of van der Corput's method, we employ two tricks that frequently appear in the literature \cite{cheng_explicit_2004, platt_improved_2015, hiary_explicit_2016, patel_explicit_2021, hiary_improved_2022, yang_explicit_2023}. First, in the $B(0, 1)$ process it is often helpful to replace the Poisson summation step with a second-derivative test that uses the Kuzmin--Landau lemma. This substitution preserves the original goal of shortening the exponential sum under consideration, without generating problematic secondary error terms that typically arise when applying Poisson summation. Second, to minimise tedium we typically apply an $A^nB$ block as a single operation instead of $n + 1$ separate operations. The following lemma, due to \cite{yang_explicit_2023}, incorporates both of these modifications, which we will make extensive use of in this work.

\begin{lemma}[Explicit $k$th derivative test]\label{kth_deriv_test} Let $a, N$ be integers with $N > 0$. Let $f(x)$ be equipped with $k \ge 3$ continuous derivatives, with $f^{(k)}$ monotonic, and suppose that $\lambda_k \le |f^{(k)}(x)| \le h\lambda_k$ for all $x \in (a, a + N]$ and some $\lambda_k > 0$ and $h > 1$. Then, for all $\eta > 0$, we have
\[
S_f(a, N) \le A_k h^{2/K}N\lambda_k^{1/(2K - 2)} + B_kN^{1 - 2/K}\lambda_k^{-1/(2K - 2)}
\]
where $K = 2^{k - 1}$, and
\[
A_3 := \sqrt{\frac{1}{\eta h} + \frac{32}{15\sqrt{\pi}}\sqrt{\eta + \lambda_0^{1/3}} + \frac{1}{3}\left(\eta + \lambda_0^{1/3}\right)\lambda_0^{1/3}}\delta_3,\quad B_3 := \frac{\sqrt{32}}{\sqrt{3}\pi^{1/4}\eta^{1/4}}\delta_3,
\]
\[
\lambda_0 := \left(\frac{1}{\eta} + \frac{32\eta^{1/2}h}{15\sqrt{\pi}}\right)^{-3}\qquad \delta_3 := \sqrt{\frac{1}{2} + \frac{1}{2}\sqrt{1 + \frac{3}{8}\pi^{1/2}\eta^{3/2}}}.
\]
and $A_k, B_k$ for $k \ge 4$ are defined recusively via
\begin{equation}\label{s1_Ak_recursive_defn}
A_{k + 1}(\eta, h) := \delta_k\left(h^{-1/K} + \frac{2^{19/12}(K - 1)}{\sqrt{(2K - 1)(4K - 3)}}A_{k}(\eta, h)^{1/2}\right),
\end{equation}
\begin{equation}
B_{k + 1}(\eta) := \delta_k\frac{2^{3/2}(K - 1)}{\sqrt{(2K - 3)(4K - 5)}} B_{k}(\eta)^{1/2},
\end{equation}
\begin{equation}
\delta_k := \sqrt{1 + \frac{2}{2337^{1 - 2/K}}\left(\frac{9\pi}{1024}\eta\right)^{1/K}}.
\end{equation}
\end{lemma}
\begin{proof}
Follows by combining \cite[Lem.\ 2.4]{yang_explicit_2023} and \cite[Lem.\ 2.5]{yang_explicit_2023}
\end{proof}
\subsection{Sources of improvement}
We briefly review the main sources of improvement of Theorem \ref{main_thm} over \eqref{patel_subweyl}, in case similar methods may be applied in other settings. Our first source of improvement originates from an interval-based argument as follows. An intermediary result in the argument of Patel \cite{patel_explicit_2021} produces a bound of the form
\[
|\zeta(1/2 + it)| \le A(t_0, t)t^{27/164},\qquad t \ge t_0
\]
where $A(t_0, t)$ is a bounded function that is decreasing in $t_0$ and increasing in $t$. This immediately implies the bound $|\zeta(1/2 + it)| \le A_0t^{27/164}$ for $t \ge t_0$, where
\[
A_0 := \lim_{t \to \infty}A(t_0, t).
\]
However, in our application $A_0$ is typically large unless we take $t_0$ to be very large, which defeats the purpose of obtaining an explicit bound holding for all $t \ge 3$. Instead, we may proceed as follows. For $t_0 < t_1 < \ldots < t_n = \infty$, we apply 
\[
A(t_j, t) \le A(t_j, t_{j + 1}),\qquad t_j \le t \le t_{j + 1}
\]
and so, for all $t \ge t_0$, 
\[
|\zeta(1/2 + it)| \le A_1 t^{27/164},\qquad A_1 := \max_{0 \le j \le n - 1} A(t_j, t_{j + 1}).
\]
The central idea is to choose the $t_j$'s so that $A(t_j, t_{j + 1})$ is never too large. For instance, we take $t_{n - 1}$ to be sufficiently large so that 
\[
\lim_{t \to\infty}A(t_{n - 1}, t)
\]
is of an acceptable size. For more details and computation, see \S \ref{sec:computation}.

A second source of improvement comes from using the improved $k$th derivative test (Lemma \ref{kth_deriv_test}) which makes use of the trivial bound to increase its sharpness. For details, see \cite[\S 2]{yang_explicit_2023}.

A third source of improvement arises from applying a sharpened version of the Poisson summation formula (i.e. using Lemma \ref{imp_poisson_summation_formula} in place of Lemma \ref{poisson_summation_formula}). In our application, the error terms introduced in estimating the stationary phase approximation to an exponential sum can be significant. 

Lastly, in bounding long exponential sums, we make use of geometrically-sized intervals so that there are $O(\log t)$ subintervals instead of $O(t^{A})$ subintervals, for some fixed $A > 0$. Since the method of proof is unable to detect cancellation between terms of two different subintervals, having less divisions is beneficial. 

\section{Improved Poisson summation formulae}
In this section we prove a sharpened version of Lemma \ref{poisson_summation_formula}, which is useful since error terms arising from Poisson summation formulae are significant in our application. The main result of this section (Lemma \ref{imp_poisson_summation_formula}) is an explicit van der Corput $B$ process. 

We begin by recalling some useful results, starting with bounds on exponential integrals. If $f'$ is continuous and $|f'(x)| \ge \lambda_1 > 0$ for $\alpha \le x \le \beta$, then by Rogers \cite[Lem. 3]{rogers_sharp_2005}
\begin{equation}\label{rogers_bound}
\left|\int_\alpha^\beta e(f(x))\,dx\right| \le \frac{1}{\pi\lambda_1}.
\end{equation}
In addition, a corollary of a result due to Kershner \cite{kershner_determination_1935, kershner_determination_1938} is that if $f''$ is continuous and $|f''(x)| \ge \lambda_2 > 0$ for $\alpha \le x \le \beta$, then 
\begin{equation}\label{titchmarsh_lem44}
\left|\int_\alpha^\beta e(f(x))\,dx\right| \le \frac{1.343}{\sqrt{\lambda_2}}.
\end{equation} 
The constant 1.343 is sharp (up to rounding) and has an exact representation in terms of Fresnel integrals, however for our purposes the decimal approximation is sufficient.\footnote{Using an arbitrary-precision numerical integration package, we find that the variables $\mu_0$ and $\gamma_0$ appearing in the main result of \cite{kershner_determination_1935} appear to be $\mu_0 = -0.7266\ldots $ and $\gamma_0 = 3.3643\ldots$ instead of the stated values $\mu_0 = -0.725\ldots$ and $\gamma_0 = 3.327\ldots$ respectively.} The same result was also proved with a constant of $4\sqrt{2/\pi} = 3.191\ldots $ in Titchmarsh \cite[Lem.\ 4.4]{titchmarsh_theory_1986} and $4/\sqrt{\pi} = 2.256\ldots$ in \cite[Eqn.\ (3)]{rogers_sharp_2005}. 

Throughout, we let $\psi(x)$ denote the digamma function, defined as the logarithmic derivative of the gamma function, i.e. $\psi(x) = \Gamma'(x) / \Gamma(x)$. We briefly recall that for $x > 0$, we have 
\begin{equation}\label{digamma_upper_bound}
\psi(x) < \log x - \frac{1}{2x}. 
\end{equation}
The digamma function has the following series representation, valid for all $x > -1$ (see e.g. \cite[\S 6.3.16]{abramowitz_handbook_2013})
\begin{equation}\label{digamma_series}
\psi(1 + x) = -\gamma + \sum_{n = 1}^{\infty}\frac{x}{n(n + x)}.
\end{equation}

Finally, we recall the following upper bound on harmonic numbers. For $x \ge 2$, we have 
\begin{equation}\label{harmonic_number_bound}
\sum_{n \le x} \frac{1}{n} = \log x + \gamma - \frac{\{x\} - 1/2}{x} + O^*\left(\frac{1}{8x^2}\right) \le \log x + \gamma + \frac{9}{32},
\end{equation}
where, here and throughout, $A = O^*(B)$ means $|A| \le B$. The equality is due to \cite{montgomery_large_1973} and the inequality follows from $x \ge 2$. 

We begin by approximating an exponential sum with a sum of exponential integrals in Lemma \ref{titchmarsh_lem_47} below, which makes explicit a result of van der Corput \cite{van_der_corput_zahlentheoretische_1921}. As a small technical detail, the range of the second sum in the below lemma is typically taken to be $(\alpha - \eta, \beta + \eta)$ for arbitrary $\eta \in (0, 1)$ --- see e.g. \cite[Lem.\ 4.7]{titchmarsh_theory_1986}. In our presentation, we fix $\eta = 1/2$, which greatly simplifies the arguments that follow. This result may be compared to \cite[Lem.\ 2.26]{patel_explicit_2021}.
\begin{lemma}\label{titchmarsh_lem_47}
Let $f(x)$ be real valued, with a continuous and steadily decreasing derivative $f'(x)$ in $(a, b)$, and let $f'(b) = \alpha, f'(a) = \beta$. Then
\begin{align*}
\left|\sum_{a < n \le b}e(f(n)) - \sum_{\alpha - \frac{1}{2} < m < \beta + \frac{1}{2}}\int_a^be(f(x) - m x)\,dx\right| \le \frac{3}{\pi}\log(\beta - \alpha + 2) + 4.
\end{align*}
\end{lemma}
\begin{proof}
Assume without loss of generality that $-\frac{1}{2} < \alpha \le \frac{1}{2}$, for otherwise we may replace $f(x)$ with $ f(x) - kx$ for a suitable integer $k$. Using Euler--Maclaurin summation (see \cite[Eqn.\ (2.12)]{titchmarsh_theory_1986}), we have 
\begin{equation}\label{sum_integral_conversion_formula}
\sum_{a < n \le b}e(f(n)) = \int_a^be(f(x))\,dx + 2\pi i \int_a^b\left(x - \lfloor x\rfloor - \frac{1}{2}\right)f'(x)e(f(x))\,dx + R(a, b)
\end{equation}
where 
\[
R(a, b) = \left(a - \lfloor a\rfloor - \frac{1}{2}\right)e(f(a)) - \left(b - \lfloor b\rfloor - \frac{1}{2}\right)e(f(b))
\]
so that $|R(a, b)| \le 1$. Meanwhile, for all non-integer $x$, we have
\[
2\pi i\left(x - \lfloor x\rfloor - \frac{1}{2}\right) = -2i\sum_{m = 1}^{\infty}\frac{\sin 2\pi m x}{m} = \sum_{m = 1}^{\infty}\frac{e(-mx) - e(mx)}{m}
\]
so that 
\begin{align*}
&2\pi i\int_a^b\left(x - \lfloor x\rfloor - \frac{1}{2}\right)f'(x)e(f(x))\,dx \\
&\qquad= \sum_{m = 1}^{\infty}\frac{1}{m}\int_a^bf'(x)e(f(x) - mx)\,dx - \sum_{m = 1}^{\infty}\frac{1}{m}\int_a^bf'(x)e(f(x) + mx)\,dx = S_1 - S_2,
\end{align*}
say. We have
\[
S_1 = \sum_{m = 1}^{\infty}\frac{1}{m}\int_a^b\frac{f'(x)}{f'(x) - m}(f'(x) - m)e(f(x) - mx)\,dx = \frac{1}{2\pi i}\sum_{m = 1}^{\infty}\frac{1}{m}\int_a^b\frac{f'(x)}{f'(x) - m}\,d(e(f(x) - mx))
\]
and similarly
\[
S_2 = \frac{1}{2\pi i}\sum_{m = 1}^{\infty}\frac{1}{m}\int_a^b\frac{f'(x)}{f'(x) + m}\,d(e(f(x) + mx)).
\]
By the second mean-value theorem, there exists $c \in (a, b)$ such that 
\[
\int_a^b\frac{f'(x)}{f'(x) + m}\,d\left(e(f(x) + mx)\right) = \frac{f'(a)}{f'(a) + m}\int_a^c\,d\left(e(f(x) + mx)\right) + \frac{f'(b)}{f'(b) + m}\int_c^b\,d\left(e(f(x) + mx)\right).
\]
However
\[
\left|\int_a^b\,d\left(e(f(x) + mx)\right)\right| = \left|e(f(b) + mb) - e(f(a) + ma)\right|\le 2
\]
so 
\[
\left|\int_a^b\frac{f'(x)}{f'(x) + m}\,d\left[e(f(x) + mx)\right]\right| \le 2\left|\frac{f'(a)}{f'(a) + m}\right| + 2\left|\frac{f'(b)}{f'(b) + m}\right| = \frac{2|\alpha|}{|m + \alpha|} + \frac{2|\beta|}{|m + \beta|}.
\]
Therefore,
\[
\pi|S_2| \le \sum_{m = 1}^{\infty}\left(\frac{|\alpha|}{m|\alpha + m|} + \frac{|\beta|}{m|\beta + m|}\right)
\]
First, since $\beta \ge \alpha > -1/2$, we have $\beta + m > 0$ and 
\begin{align}
\sum_{m = 1}^{\infty}\frac{|\beta|}{m|\beta + m|} = \left|\sum_{m = 1}^{\infty}\frac{\beta}{m(\beta + m)}\right| &\le \max\{-(\psi(1/2) + \gamma), \psi(\beta + 1) + \gamma\} \notag \\
&< \log(\beta + 1) + 3\log 2, \label{poisum_beta_sum_bound}
\end{align}
where the first inequality follows from \eqref{digamma_series} for $\beta \ge 0$ and via a separate evaluation for $-1/2 < \beta < 0$. Similarly, since $|\alpha| \le 1/2$,
\[
\sum_{m = 1}^{\infty}\frac{|\alpha|}{m|\alpha + m|} \le \max\{-(\psi(1/2) + \gamma), \psi(\alpha + 1) + \gamma\} \le 2\log 2
\]
hence
\begin{equation}\label{poisson_summation_S2_bound}
|S_2| \le \frac{\log(\beta + 1) + 5\log 2}{\pi}.
\end{equation}
Now consider $S_1$. Let $M := \max\{1, \beta + 1/2\}$ and
\begin{equation}\label{poisson_summation_S1_split}
S_1 = \sum_{1 \le m < M} + \sum_{m \ge M} = S_{11} + S_{12}.
\end{equation}
We have
\[
\pi|S_{12}| \le \sum_{m \ge M}\left(\frac{|\alpha|}{m(m - \alpha)} + \frac{|\beta|}{m(m - \beta)}\right).
\]
If $n = \lfloor m - \beta \rfloor$, then $m(m - \beta) \ge n(n + \beta)$. Furthermore note that $n \ge 1$ for all $m \ge M + 1$, and that there is one integer in $[M, M + 1)$, say $m_0$. Therefore, since $\beta > -1/2$ and by \eqref{poisum_beta_sum_bound}, 
\begin{equation}\label{poisum_subtracted_series_bound}
\sum_{m \ge M}\frac{|\beta|}{m(m - \beta)} \le \frac{|\beta|}{m_0(m_0 - \beta)} + \sum_{n = 1}^{\infty}\frac{|\beta|}{n(n + \beta)}
\end{equation}
If $\beta < 0$, then $m_0(m_0 - \beta) \ge M(M - \beta) > 1$, and hence the first term on the RHS is at most $1/2$. Meanwhile using the same argument as \eqref{poisum_beta_sum_bound}, the sum on the RHS is bounded by $-\psi(1/2) - \gamma = 2\log 2$. On the other hand if $\beta \ge 0$, then by the same argument used in \eqref{poisum_beta_sum_bound}
\[
\sum_{m \ge M}\frac{|\beta|}{m(m - \beta)} \le \frac{\beta}{\frac{1}{2}(\beta + \frac{1}{2})} + \psi(\beta + 1) + \gamma \le \log(\beta + 1) + \gamma + 2.
\]
In either case the RHS of \eqref{poisum_subtracted_series_bound} is at most $\log(\beta + 1) + 1/2 + 3\log 2$. Similarly, writing $n' = \lfloor m - \alpha \rfloor$, so that $n' \ge 1$ for all $m \ge M + 1$ (since $\beta \ge \alpha$), and using $|\alpha| \le 1/2$,
\begin{align*}
\sum_{m \ge M}\frac{|\alpha|}{m(m - \alpha)} &\le \frac{|\alpha|}{m_0(m_0 - \alpha)} + \sum_{n' = 1}^{\infty}\frac{|\alpha|}{n'(n' + \alpha)} \\
&\le \max\{1/2 + 2\log 2, 1 + \psi(\alpha + 1) + \gamma \} \le 1/2 + 2\log 2.
\end{align*}
Thus
\begin{equation}\label{S_12_bound}
|S_{12}|\le \frac{\log(\beta + 1) + 1 + 5\log 2}{\pi}.
\end{equation}
We now divide our argument into the following two cases. 
\subsubsection*{Case 1: $\beta \le 1/2$}
Then, $M = 1$ and $S_{11}$ is an empty sum. Then, we have (vacuously) 
\begin{equation}\label{poisum_case1_bound}
\left|S_1 - \sum_{1 \le m < \beta + 1/2}\int_a^be(f(x) - mx)\,dx\right| = |S_{12}|,
\end{equation}
since the sum on the LHS is empty. 
\subsubsection*{Case 2: $\beta > 1/2$}
Then, $M = \beta + 1/2$ and we let 
\begin{equation}\label{poisson_summation_S11}
S_{11} = S_3 + \sum_{1 \le m < \beta + 1/2}\int_a^be(f(x) - mx)\,dx,
\end{equation}
where 
\begin{align*}
S_3 &:= \sum_{1 \le m < \beta + 1/2}\frac{1}{m}\int_a^b(f'(x) - m)e(f(x) - mx)\,dx\\
&= \left|\sum_{1 \le m < \beta + 1/2}\frac{1}{2\pi m i}e(f(x) - mx)\right|_a^b.
\end{align*}
Therefore,
\[
|S_3| \le \frac{1}{\pi}\sum_{1 \le m < \beta + 1/2}\frac{1}{m} < \frac{1}{\pi}\left(\log (\beta + 1) + \gamma + \frac{9}{32}\right),
\]
where in the last inequality we have used \eqref{harmonic_number_bound} if $\beta + 1/2 \ge 2$, and a direct evaluation if $\beta + 1/2 < 2$. It follows that in this case, from combining \eqref{poisson_summation_S1_split} and \eqref{poisson_summation_S11}, that 
\begin{equation}\label{poisum_case2_bound}
\left|S_1 - \sum_{1 \le m < \beta + 1/2}\int_a^be(f(x) - mx)\,dx\right| \le |S_3| + |S_{12}|
\end{equation}

Therefore, in either case, by collecting \eqref{sum_integral_conversion_formula}, \eqref{poisson_summation_S2_bound}, \eqref{S_12_bound}, \eqref{poisum_case1_bound} and \eqref{poisum_case2_bound} we have 
\begin{align*}
\sum_{a < n \le b}e(f(n)) &= \sum_{0 \le m < \beta + 1/2}\int_a^be(f(x) - mx)\,dx + R_1
\end{align*}
where 
\begin{align*}
|R_1| &\le |S_{2}| + |S_{12}| + |S_3| + 1\\
&\le \frac{\log(\beta + 1) + 5\log 2}{\pi} + \frac{\log(\beta + 1) + 1 + 5\log 2}{\pi} \\
&\qquad\qquad + \frac{1}{\pi}\left(\log(\beta + 1) + \gamma + \frac{9}{32}\right) + 1\\
&< \frac{3}{\pi}\log(\beta + 1) + 4.
\end{align*}
To complete the argument we note that the assumption that $-1/2 < \alpha \le 1/2$ implies that $\beta + 1 < \beta - \alpha + 2$,\footnote{This inequality can be readily sharpened (the constant of 2 is chosen for cosmetic purposes). In any case, for our application the constant term makes no difference to the final result.} and that sums over $[0, \beta + 1/2)$ are equivalent to sums over $(\alpha - 1/2, \beta + 1/2)$.
\end{proof}

Next, we require a lemma related to the principle of stationary phase, which approximates an exponential integral. The traditional presentation of this result (see e.g. \cite[Lem.\ 4.6]{titchmarsh_theory_1986}) has a main error term of size $O(\lambda_2^{-4/5}\lambda_3^{1/5})$, where $\lambda_2$, $\lambda_3$ are respectively the orders of the second and third derivative of the phase function. In the following lemma, we make explicit an argument of Phillips \cite{phillips_zetafunction_1933} to bound this error term to $O(\lambda_2^{-1}\lambda_3^{1/3})$, which is smaller in our application. We also record that Heath-Brown \cite{heath-brown_pjateckii-sapiro_1983} has shown that under suitable conditions, the main error term may be removed completely. However, since the error term is already of an acceptable size, we do not pursue such an optimisation here. 

\begin{lemma}\label{titchmarsh_lem_46}
Let $f(x)$ be real and three times differentiable, satisfying $f'' < 0$, 
\[
0 < \lambda_2 \le |f''(x)| \le h_2\lambda_2,\qquad |f'''(x)| \le h_3\lambda_3,\qquad x \in (a, b).
\]
Furthermore, suppose $f'(c) = 0$ for some $c \in [a, b]$. Then,
\begin{align*}
\left|\int_a^be(f(x))\,dx - \frac{e(f(c) - 1/8)}{|f''(c)|^{1/2}}\right| &\le \frac{2\cdot 3^{2/3}}{\pi^{2/3}}h_3^{1/3}\lambda_2^{-1}\lambda_3^{1/3} + \frac{1}{\pi}\left(\frac{1}{|f'(a)|} + \frac{1}{|f'(b)|}\right).
\end{align*}
\end{lemma}
\begin{proof}
Suppose first that $c \in [a + \delta, b - \delta]$, for some fixed $\delta > 0$ to be chosen later. Let
\begin{equation}\label{I1I2I3_defn}
\int_a^be(f(x))\,dx = \int_a^{c - \delta} + \int_{c - \delta}^{c + \delta} + \int_{c + \delta}^b = I_1 + I_2 + I_3,
\end{equation}
say. Since $f'' < 0$, for all $x \in [a, c - \delta]$ we have by the mean-value theorem 
\[
f'(x) \ge f'(c - \delta) = f'(c) - \delta f''(\xi) \ge \delta\lambda_2
\]
for some $\xi \in [c - \delta, c]$. Similarly, $|f'(x)| \ge \delta\lambda_2$ for all $x\in [c + \delta, b]$. Via \eqref{rogers_bound}, we have 
\begin{equation}\label{statphase_endpoints_bound}
|I_1|, |I_3| \le \frac{1}{\pi\delta\lambda_2}.
\end{equation}
Let 
\[
g(x) := f(x + c) - f(c) - \frac{1}{2}x^2f''(c)
\]
so that, for all $x$ there exists some $\xi \in (c, x + c)$ such that
\begin{equation}
g(x) = f(c) + xf'(c) + \frac{x^2}{2}f''(c) + \frac{x^3}{6}f'''(\xi) - f(c) - \frac{x^2}{2}f''(c) \le \frac{x^3}{6}h_3\lambda_3. \label{boundg(x)}
\end{equation}
Hence
\begin{align}
\int_{c - \delta}^{c + \delta}e(f(x))\,dx &= e(f(c))\int_{c - \delta}^{c + \delta}e\left(\frac{(x - c)^2}{2}f''(c)\right)e\left(f(x) - f(c) - \frac{(x - c)^2}{2}f''(c)\right)\,dx\notag\\
&= e(f(c))\int_{-\delta}^{\delta}e\left(\frac{x^2}{2}f''(c)\right)\,dx + e(f(c))\int_{-\delta}^{\delta}e\left(\frac{x^2}{2}f''(c)\right)\left(e(g(x)) - 1\right)\,dx.\label{statphase_middleterm_bound}
\end{align}
However,
\begin{align*}
\int_{-\delta}^{\delta}e\left(\frac{x^2}{2}f''(c)\right)\,dx &= \frac{1}{|\pi f''(c)|^{1/2}}\int_0^{\pi\delta^2|f''(c)|}\frac{e^{-iu}}{\sqrt{u}}\,du\\
&= \frac{1}{|\pi f''(c)|^{1/2}}\left(\int_0^{\infty}\frac{e^{-iu}}{\sqrt{u}}\,du - \int_{\pi\delta^2|f''(c)|}^{\infty}\frac{e^{-iu}}{\sqrt{u}}\,du\right)\\
&= \frac{e(-1/8)}{|\pi f''(c)|^{1/2}} + O^*\left(\frac{1}{\pi \delta |f''(c)|}\right)
\end{align*}
so that 
\begin{equation}\label{statphase_mainterm_bound}
\left|e(f(c))\int_{-\delta}^{\delta}e\left(\frac{x^2}{2}f''(c)\right)\,dx - \frac{e(f(c) -1/8)}{|f''(c)|^{1/2}}\right| \le \frac{1}{\pi\delta \lambda_2}.
\end{equation}
We will now bound the modulus of
\[
I := \int_{-\delta}^{\delta}e\left(\frac{x^2}{2}f''(c)\right)\left(e(g(x)) - 1\right)\,dx.
\]
First, suppose that $\delta \le h/(\lambda_2\delta)$ for some arbitrary constant $h > 0$ to be chosen later. Then via the trivial bound, we have
\begin{equation}\label{statphase_trivial_bound_case}
|I| \le \int_{-\delta}^{\delta}|e(g(x)) - 1|\,dx \le 4\delta \le \frac{4h}{\delta\lambda_2}.
\end{equation}
Assume now that $\delta > h/(\lambda_2\delta)$. Then
\[
I = \int_{-\delta}^{-h/(\lambda_2\delta)} + \int_{-h/(\lambda_2\delta)}^{h/(\lambda_2\delta)} + \int_{h/(\lambda_2\delta)}^{\delta} = I_4 + I_5 + I_6,
\]
say. First, consider $I_5$. Using the trivial bound, we have 
\begin{equation}\label{statphase_I2_bound}
|I_5| \le \int_{-h/(\lambda_2\delta)}^{h/(\lambda_2\delta)}|e(g(x)) - 1|\,dx \le \frac{4h}{\delta\lambda_2}.
\end{equation}
Next, consider $I_6$. Letting $\lambda = f''(c) / 2$, and integrating by parts, we have 
\begin{align}
I_6 &= \int_{h/(\lambda_2\delta)}^{\delta}e(\lambda x^2)(e(g(x)) - 1)\,dx = \int_{h/(\lambda_2\delta)}^{\delta}4\pi i \lambda x e^{2\pi i \lambda x^2}\frac{e(g(x)) - 1}{4\pi i \lambda x}\,dx  \notag\\
&= \left[e(\lambda x^2)\frac{e(g(x)) - 1}{4\pi i \lambda x}\right]_{h/(\lambda_2\delta)}^{\delta} - \int_{h/(\lambda_2\delta)}^{\delta}e(\lambda x^2)\frac{\,d}{\,dx}\left(\frac{e(g(x)) - 1}{4\pi i \lambda x}\right)\,dx \label{statphase_I3_intermediate_bound}
\end{align}
However, 
\[
\frac{\,d}{\,dx}\left(\frac{e^{2\pi i g(x)} - 1}{x}\right) = \frac{2\pi i x g'(x) e(g(x)) - (e(g(x)) - 1)}{x^2}
\]
and for some $\xi \in (c, x + c)$, we have 
\begin{equation}\label{statphase_gpx_bound}
g'(x) = f'(x + c) - x f''(c) = \left(f'(c) + xf''(c) + \frac{x^2}{2}f'''(\xi)\right) - xf''(c) \le \frac{x^2}{2}h_3\lambda_3.
\end{equation}
Furthermore, we use the identity
\[
e(x) = \int_0^{2\pi x}ie^{it}\,dt + 1
\]
to obtain using (\ref{boundg(x)}) that
\begin{equation}\label{statphase_egx_bound}
|e(g(x)) - 1| = \left|\int_0^{2\pi g(x)}ie^{it}\,dt\right| \le 2\pi g(x) \le \frac{\pi}{3}x^3h_3\lambda_3.
\end{equation}
This implies that, by combining \eqref{statphase_gpx_bound} and \eqref{statphase_egx_bound},
\[
\left|2\pi i x g'(x) e(g(x)) - (e(g(x)) - 1)\right| \le 2\pi \frac{x^3}{2}h_3\lambda_3 + \frac{\pi}{3}x^3h_3\lambda_3 = \frac{4\pi}{3}x^3h_3\lambda_3,
\]
and thus
\begin{equation}\label{statphase_part_sum2_I3_bound}
\left|\int_{h/(\lambda_2\delta)}^{\delta}e(\lambda x^2)\frac{\,d}{\,dx}\left(\frac{e(g(x)) - 1}{4\pi i \lambda x}\right)\,dx\right| \le \frac{h_3\lambda_3}{3|\lambda|}\int_{h/(\lambda_2\delta)}^{\delta}x\,dx = \frac{h_3\lambda_3}{6|\lambda|}\left(\delta^2 - \frac{h^2}{(\delta\lambda_2)^2}\right).
\end{equation}
Meanwhile, once again using \eqref{statphase_egx_bound}, and the triangle inequality,
\begin{equation}\label{statphase_part_sum_I3_bound}
\left|\left[e(\lambda x^2)\frac{e(g(x)) - 1}{4\pi i \lambda x}\right]_{h/(\lambda_2\delta)}^{\delta}\right| \le \frac{h_3\lambda_3}{12|\lambda|}\left(\delta^2 + \frac{h^2}{(\delta\lambda_2)^2}\right),
\end{equation}
and so, collecting \eqref{statphase_I3_intermediate_bound}, \eqref{statphase_part_sum2_I3_bound} and \eqref{statphase_part_sum_I3_bound}, and using $|\lambda| = |f''(c)| / 2 \ge \lambda_2 / 2$,
\begin{equation}\label{statphase_I3_bound}
|I_6| \le \frac{h_3\lambda_3}{4|\lambda|}\delta^2 - \frac{h_3\lambda_3 h^2}{12(\delta\lambda_2)^2} < \frac{h_3}{2}\frac{\lambda_3}{\lambda_2}\delta^2.
\end{equation}
We bound $I_4$ in the same way. Therefore, collecting \eqref{statphase_I2_bound} and \eqref{statphase_I3_bound} we find
\begin{equation}\label{statphase_I_bound_nontrivial_case}
|I| \le \frac{h_3\lambda_3\delta^2}{\lambda_2} + \frac{4h}{\delta\lambda_2},
\end{equation}
in the case where $\delta > h/(\lambda_2\delta)$. However, since the bound \eqref{statphase_I_bound_nontrivial_case} is strictly greater than \eqref{statphase_trivial_bound_case}, we conclude \eqref{statphase_I_bound_nontrivial_case} in fact holds for all $\delta > 0$. Combining this with \eqref{statphase_endpoints_bound}, \eqref{statphase_middleterm_bound}, \eqref{statphase_mainterm_bound}, we find that 
\[
\left|\int_a^b e(f(x))\,dx - \frac{e(f(c) - 1/8)}{|f''(c)|^{1/2}}\right| \le \frac{h_3\lambda_3\delta^2}{\lambda_2} + \left(4h + \frac{3}{\pi}\right)\frac{1}{\lambda_2\delta}.
\]
However, since $h > 0$ is arbitrary, we take the limit as $h \to 0^+$ and choose 
\[
\delta = \left(\frac{3}{\pi h_3}\right)^{1/3}\lambda_3^{-1/3}
\]
to balance the two terms on the RHS. This choice gives
\begin{equation}\label{statphase_final_eqn_1}
\left|\int_a^b e(f(x))\,dx - \frac{e(f(c) - 1/8)}{|f''(c)|^{1/2}}\right| \le \frac{2\cdot 3^{2/3}}{\pi^{2/3}}h_3^{1/3}\lambda_2^{-1}\lambda_3^{-1/3}.
\end{equation}
If $c + \delta > b$, then we instead bound $I_3$ using
\begin{equation}\label{statphase_final_eqn_2}
|I_3| = \left|\int_b^{c + \delta}e(f(x))\,dx\right| \le \frac{1}{\pi|f'(b)|},
\end{equation}
which follows from \eqref{rogers_bound} since for all $x \in [b, c + \delta]$, we have 
$0 = f'(c) > f'(b) \ge f'(x)$, as $b > c$ and $f'' < 0$. Similarly, if $c - \delta < a$, we instead bound $I_1$ using
\begin{equation}\label{statphase_final_eqn_3}
|I_1| = \left|\int_{c - \delta}^ae(f(x))\,dx\right| \le \frac{1}{\pi|f'(a)|}.
\end{equation}
The desired result follows from combining \eqref{statphase_final_eqn_1}, \eqref{statphase_final_eqn_2} and \eqref{statphase_final_eqn_3}.
\end{proof}

\begin{lemma}[Improved Poisson summation formula]\label{imp_poisson_summation_formula}
Let $f(x)$ be three times differentiable. Let $f'(x)$ be decreasing in $[a, b]$ and $f'(b) = \alpha$, $f'(a) = \beta$. For all integer $\nu \in (\alpha, \beta]$, let $x_\nu$ be defined by $f'(x_\nu) = \nu$. Furthermore suppose that $\lambda_2 \le |f''(x)| \le h_2\lambda_2$ and $|f'''(x)| \le h_3\lambda_3$. Then 
\begin{align*}
&\left|\sum_{a < n \le b}e(f(n)) - \sum_{\alpha < \nu \le \beta}\frac{e(f(x_{\nu}) - \nu x_\nu-1/8)}{|f''(x_\nu)|^{1/2}}\right| \\
&\qquad\qquad\qquad\qquad\le \frac{4.686}{\sqrt{\lambda_2}} + \frac{2\cdot 3^{2/3}}{\pi^{2/3}}h_2 h_3^{1/3}(b - a)\lambda_3^{1/3} + \frac{5}{\pi}\log(\beta - \alpha + 2) + 6.
\end{align*}
\end{lemma}
\begin{proof}
We use Lemma \ref{titchmarsh_lem_46} to obtain
\begin{align}
&\sum_{\alpha + \frac{1}{2} < \nu < \beta - \frac{1}{2}}\int_a^be(f(x) - \nu x)\,dx = \sum_{\alpha + \frac{1}{2} < \nu < \beta - \frac{1}{2}}\frac{e(f(x_{\nu})- \nu x_{\nu} - 1/8)}{|f''(x_{\nu})|^{1/2}} \notag\\
&\qquad + \sum_{\alpha + \frac{1}{2} < \nu < \beta - \frac{1}{2}}\frac{2\cdot 3^{2/3}}{\pi^{2/3}}h_3^{1/3}\lambda_2^{-1}\lambda_3^{1/3} + \sum_{\alpha + \frac{1}{2} < \nu < \beta - \frac{1}{2}}\frac{1}{\pi}\left(\frac{1}{|f'(a) - \nu|} + \frac{1}{|f'(b) - \nu|}\right) \notag\\
&\qquad\qquad\qquad\qquad= S_1 + S_2 + S_3. \label{poisum_final_mainsum}
\end{align}
Since there is at most one integer each in the intervals $(\alpha, \alpha + \frac{1}{2}]$ and $[\beta - \frac{1}{2}, \beta]$, and $|f''(x_{\nu})|^{1/2} \ge \lambda_2^{1/2}$, we have 
\begin{equation}\label{poisum_final_S1_bound}
S_1 = \sum_{\alpha < \nu \le \beta}\frac{e(f(x_{\nu}) - \nu x_{\nu} - 1/8)}{|f''(x_{\nu})|^{1/2}} + O^*\left(\frac{2}{\sqrt{\lambda_2}}\right).
\end{equation}
Next, since $|f''(x)| \le h_2\lambda_2$, we have $\beta - \alpha = f'(a) - f'(b) \le (b - a)h_2\lambda_2$,
\begin{align}
|S_2| < \frac{2\cdot 3^{2/3}}{\pi^{2/3}}h_3^{1/3}(\beta - \alpha)\lambda_2^{-1}\lambda_3^{1/3} \le \frac{2\cdot 3^{2/3}}{\pi^{2/3}}h_2 h_3^{1/3}(b - a)
\lambda_3^{1/3}.\label{poisum_final_S2_bound}
\end{align}
Finally, consider $S_3$. For all $a \le \nu \le b$, we have $|f'(a) - \nu| = \beta - \nu$ and $|f'(b) - \nu| = \nu - \alpha$. Furthermore, the $k$th smallest integer in the interval $(\alpha + 1/2, \beta + 1/2)$ is bounded below by $\alpha + k - 1/2$. Therefore,
\begin{align*}
\sum_{\alpha + \frac{1}{2} < \nu < \beta - \frac{1}{2}}\frac{1}{\nu - \alpha} < \sum_{1 \le n < \beta - \alpha - \frac{1}{2}}\frac{1}{n - \frac{1}{2}} &< \psi(\beta - \alpha) - \psi(1/2)
\end{align*}
where $\psi(x)$ is the digamma function. Similarly 
\[
\sum_{\alpha + \frac{1}{2} < \nu < \beta - \frac{1}{2}}\frac{1}{\beta - \nu} < \psi(\beta - \alpha) - \psi(1/2).
\]
Therefore, using $\psi(x) \le \log x$ for $x > 0$,
\begin{equation}\label{poisum_final_S3_bound}
|S_3| \le \frac{2}{\pi}\psi(\beta - \alpha) - \frac{2}{\pi}\psi(1/2) < \frac{2}{\pi}\log(\beta - \alpha) + 2. 
\end{equation}
Finally, since the intervals $(\alpha - 1/2, \alpha + 1/2]$ and $[\beta - 1/2, \beta + 1/2)$ contain at most one integer each, and using \eqref{titchmarsh_lem44},
\begin{equation}\label{poisum_final_sum_offset}
\sum_{\alpha + \frac{1}{2} < \nu < \beta - \frac{1}{2}}\int_a^be(f(x) - \nu x)\,dx = \sum_{\alpha - \frac{1}{2} < \nu < \beta + \frac{1}{2}}\int_a^be(f(x) - \nu x)\,dx + O^*\left(\frac{2.686}{\sqrt{\lambda_2}}\right).
\end{equation}
The desired result follows upon applying Lemma \ref{titchmarsh_lem_47} and collecting \eqref{poisum_final_mainsum}, \eqref{poisum_final_S1_bound}, \eqref{poisum_final_S2_bound}, \eqref{poisum_final_S3_bound} and \eqref{poisum_final_sum_offset}.
\end{proof}

\section{Proof of Theorem \ref{main_thm}}
This section contains the proof of our main result. We derive an upper bound on $\zeta(1/2 + it)$ using the Riemann--Siegel formula, which allows us to express $\zeta(1/2 + it)$ in terms of an exponential sum of length $O(t^{1/2})$. This enables us to readily apply the techniques of the previous sections to produce a non-trivial estimate of $\zeta(1/2 + it)$. The main ingredients this step are the explicit $A$, $B$ and $A^kB(0, 1)$ processes, given by Lemma \ref{weyl_differencing}, \ref{imp_poisson_summation_formula} and \ref{kth_deriv_test} respectively. 

To begin, we recall the following result, due to Hiary \cite{hiary_explicit_2016}, which is a consequence of the Riemann--Siegel formula.
\begin{lemma}\label{riemann_siegel_formula}
For all $t \ge 200$, 
\begin{equation}\label{riemann_siegel_bound}
|\zeta(1/2 + it)| \le 2\left|\sum_{1 \le n \le \lfloor \sqrt{t/(2\pi)}\rfloor}n^{-1/2 - it}\right| + 1.48t^{-1/4} + 0.127 t^{-3/4}
\end{equation}
\end{lemma}
\begin{proof}
See Hiary \cite[Lem.\ 2.1]{hiary_explicit_2016} and also Gabcke \cite{gabcke_neue_1979}.
\end{proof}

We use a similar approach as \cite[Thm. 5.18]{titchmarsh_theory_1986} to evaluate the main sum on the RHS of \eqref{riemann_siegel_bound}. We divide the sum into three subsums and bound each individually. Firstly, for $n \ll t^{27/82}$, we use the triangle inequality and the trivial bound. Secondly, for $t^{27/82} \ll n \ll t^{7/17}$ we use Lemma \ref{kth_deriv_test} with $k = 4$. Lastly, for $t^{7/17} \ll n \ll t^{1/2}$, we use the $ABA^3B(0, 1)$ process (Lemma \ref{exponential_sum_lemma} below). 

We remark that the last subsum, taken over $t^{7/17} \ll n \ll t^{1/2}$, is by far the most significant. In fact, the second sum can be bounded to be $\ll t^{19/119}$ which is $o(t^{27/164})$. Additionally, we have the freedom make the first subsum as small as we please, by appropriately choosing the boundary between the first and second subsums. Therefore, in what follows we will expend the most effort in bounding the third subsum.

To begin, let $h_1, h_2 > 1$, and $\theta_1$, $\theta_2$, $\theta_3 > 0$ be scaling parameters to be chosen later. Define
\begin{equation}\label{a_k_defn}
a_k := \left\lfloor h_1^{-k} \sqrt{\frac{t}{2\pi}} \right\rfloor, \qquad k = 0, 1, \ldots, K
\end{equation}
where 
\begin{equation}\label{K_defn}
K := K(t) = \left\lceil\frac{\frac{3}{34}\log t - \log(\theta_2\sqrt{2\pi})}{\log h_1}\right\rceil.
\end{equation}
Note that this choice of $K$ guarantees that 
\begin{equation}\label{a_k_bound}
a_K \le h_1^{-K}\sqrt{\frac{t}{2\pi}} \le \theta_2 t^{7/17}.
\end{equation}
Similarly, let 
\begin{equation}\label{a_r_bound}
a'_r = \left\lfloor h_2^{-r}\theta_2t^{7/17}\right\rfloor,\qquad r = 0, 1, \ldots R,
\end{equation}
\begin{equation}\label{R_defn}
R := R(t) = \left\lceil \frac{\frac{115}{1394}\log t - \log(\theta_3/\theta_2)}{\log h_2}\right\rceil
\end{equation}
so that $a_0' \le a_K$. These parameters are chosen so that $a_R' \le \theta_3t^{27/82}$.

We thus divide 
\begin{equation}
\begin{split}
\sum_{1 \le n \le \lfloor\sqrt{t/(2\pi)}\rfloor} n^{-1/2 - it} &= \sum_{1 \le n \le a'_R} + \sum_{a'_R < n \le a_K} + \sum_{a_K < n \le \lfloor\sqrt{t/(2\pi)}\rfloor}\\
&= S_1 + S_2 + S_3,
\end{split}
\end{equation}
say. The next few lemmas are used to bound each of the three subsums. 
\begin{lemma}\label{S1_bound_lem}
For $t \ge t_0$ and $\theta_3 > 0$, we have 
\[
|S_1| \le C_{0}t^{27/164} - \sqrt{2},\qquad C_{0} := 2\sqrt{\theta_3\left(1 + \frac{1}{2t_0^{27/82}}\right)}.
\]
\end{lemma}
\begin{proof}
Recall that $a_R' \le \theta_3t^{27/82}$, so that, by the triangle inequality and the trivial bound,
\begin{align*}
\left|\sum_{1 \le n \le a'_R}n^{-1/2 - it}\right| \le \sum_{n = 1}^{a'_R}\frac{1}{\sqrt{n}} &\le \int_{1/2}^{a'_R + 1/2}\frac{\,dx}{x^{1/2}} \le \int_{1/2}^{\theta_3t^{27/82} + 1/2}\frac{\,dx}{x^{1/2}} \\
&\le 2\sqrt{\theta_3t^{27/82} + \frac{1}{2}} - \sqrt{2} \le C_{0}t^{27/164} - \sqrt{2},
\end{align*}
for all $t \ge t_0$. Here, the second inequality follows from the convexity of $x^{-1/2}$ and Jensen's inequality, since
\[
n^{-1/2} = \left(\int_{n - 1/2}^{n + 1/2}x\,dx\right)^{-1/2} \le \int_{n - 1/2}^{n + 1/2}\frac{\,dx}{x^{1/2}}.
\]
\end{proof}

\begin{lemma}\label{region_2_bound}
Suppose $0 < t_0 \le t \le t_1$, $h_2 > 1$ and $\eta_2, \theta_2, \theta_3 > 0$. Then
\[
|S_2| \le D_3 t^{19/119} + D_4t^{71/476}.
\]
where
\[
h_3 := \frac{h_2}{1 - \frac{h_2}{\theta_2}t_0^{-27/82}},
\]
\[
D_3 := \frac{3^{1/14}}{\pi^{1/14}}A_4(\eta_2, h_3) h_3^{5/7}(h_3 - 1)\theta_2^{3/14}\frac{1 - h_2^{-3R(t_1)/14}}{h_2^{3/14} - 1}, 
\]
\[
D_4 := \frac{\pi^{1/14}}{3^{1/14}}B_4(\eta_2) h_3^{2/7}(h_3 - 1)^{3/4}\theta_2^{15/28}\frac{1 - h_2^{-15R(t_1)/28}}{h_2^{15/28} - 1}.
\]
\end{lemma}
\begin{proof}
With $a'_r$ as defined in \eqref{a_r_bound}, we have 
\[
\frac{a_r'}{a'_{r + 1}} \le \frac{h_2^{-r}\theta_2t^{7/17}}{h_2^{-(r + 1)}\theta_2t^{7/17} - 1} \le h_3.
\]
since 
\begin{align*}
h_2^{-(r + 1)}\theta_2t^{7/17} &\ge h_2^{-(\frac{115}{1394}\log t - \log(\theta_3/\theta_2))/\log h_2 - 1}\theta_2t^{7/17} \\
&= \frac{\theta_2}{h_2}t^{-115/1394}\frac{\theta_3}{\theta_2}\theta_2t^{7/17} \ge\frac{\theta_3}{h_2}t_0^{27/82}
\end{align*}
Applying Lemma \ref{kth_deriv_test} with $k = 4$, $h = h_3^4$, $a = a'_r$, $b = a'_{r - 1} \le h_3a$ and 
\[
f(x) = -\frac{t}{2\pi}\log x,\qquad \lambda_4 = \frac{3t}{\pi (h_3a'_r)^4},
\]
we obtain, for any $\eta_2 > 0$, 
\begin{align*}
S_f(a_{r}', a_{r - 1}' - a_{r}') &\le A_4(\eta_2, h_3^4)h_3(h_3 - 1)a'_r\left(\frac{3t}{\pi (h_3a_r')^4}\right)^{1/14} \\
&\qquad\qquad + B_4(\eta_2)(h_3 - 1)^{3/4}{a'_r}^{3/4}\left(\frac{3t}{\pi (h_3a'_r)^4}\right)^{-1/14}\\
&= D_1a_r'^{5/7}t^{1/14} + D_2a_r'^{29/28}t^{-1/14}
\end{align*}
where 
\[
D_1(h_3) := \frac{3^{1/14}}{\pi^{1/14}}A_4h_3^{5/7}(h_3 - 1),\qquad D_2(h_3) := \frac{\pi^{1/14}}{3^{1/14}}B_4h_3^{2/7}(h_3 - 1)^{3/4}.
\]
Next, by partial summation 
\begin{align*}
\left|\sum_{a'_r < n \le a'_{r - 1}}n^{-1/2 - it}\right| &\le {a'_r}^{-1/2}\max_{L \in (a'_r, a'_{r - 1}]}S_f(a_r', L - a_r') \\
&\le D_1(h_3) {a'_r}^{3/14}t^{1/14} + D_2(h_3){a'_r}^{15/28}t^{-1/14}
\end{align*}
so that, combined with
\[
a'_r \le h_2^{-r}\theta_2t^{7/17},
\]
we obtain
\begin{align*}
\left|\sum_{a'_R < n \le a'_0}n^{-1/2 - it}\right| &\le \sum_{r = 1}^R\left|\sum_{a'_r < n \le a'_{r - 1}}n^{-1/2 - it}\right|\\
&\le D_1t^{1/14}\sum_{r = 1}^R\left(h_2^{-r}\theta_2t^{7/17}\right)^{3/14} + D_2t^{-1/14}\sum_{r = 1}^R\left(h_2^{-r}\theta_2t^{7/17}\right)^{15/28}\\
&= D_1\theta_2^{3/14}t^{19/119}\sum_{r = 1}^Rh_2^{-3r/14} + D_2\theta_2^{15/28}t^{71/476}\sum_{r = 1}^Rh_2^{-15r/28}\\
&= D_1\theta_2^{3/14}\frac{1 - h_2^{-3R/14}}{h_2^{3/14} - 1}t^{19/119} + D_2\theta_2^{15/28}\frac{1 - h_2^{-15R/28}}{h_2^{15/28} - 1}t^{71/476}.
\end{align*}
The result follows from $R(t) \le R(t_1)$. 
\end{proof}

\begin{lemma}\label{exponential_sum_lemma}
Let $t \ge t_0 > 0$ and $\theta_1, \theta_2, \eta_1 > 0$, $1 < h \le 2$ be arbitrary constants. Suppose $a, b$ satisfy $\theta_2t^{7/17} < a < b \le ha$ and $q_0 := \theta_1\theta_2^{7/17}t_0^{65/697} \ge 2$. Then
\begin{align*}
\left|\sum_{a < n \le b}n^{-it}\right| &\le C_1(h)a^{23/41}t^{11/82} + C_2(h)a^{147/328}t^{61/328} \\
&\qquad\qquad + E_1(h)a^{169/164}t^{-15/82} + E_2(h)a^{59/123}t^{5/41} + E_3(h)a^{1/2}.
\end{align*}
where
\[
C_1(h) := \alpha\left(\frac{\theta_1^{-1}(h - 1)}{1 - q_0^{-1}} + 0.4750\,\theta_1^{11/30}A_5\left(\eta_1, \frac{76545\sqrt{2}}{107264} h^9\right)h^{21/8}(h - 1)\right)^{1/2},
\]
\[
C_2(h) := \alpha\left(0.2531\,\theta_1^{61/120}B_5(\eta_1)h^{3/2}(h - 1)^{7/8}\right)^{1/2},
\]
\[
E_1(h) := \alpha(12.496\sqrt{\pi})^{1/2}h^{3/4}\left(\theta_1 - \frac{1}{\theta_2t_0^{100/697}}\right)^{-1/4},
\]
\[
E_2(h) := \alpha\left(\frac{9}{14}\theta_1^{1/3}\left(\frac{4.465(h - 1)^{1/3}}{\pi^{4/3}\theta_2^{1/3}t_0^{7/51}} + \frac{6}{\pi}h^{3}(h - 1)\right)\right)^{1/2},
\]
\[
E_3(h) := \alpha\left(6 + \frac{5}{\pi}\log 2\right)^{1/2},\qquad\alpha = \sqrt{h - 1 + \frac{\theta_1}{\theta_2^{5/41}t_0^{222/697}}},
\]
and $A_5, B_5$ are functions defined in Lemma \ref{kth_deriv_test}.
\end{lemma}
\begin{proof}
Let
\[
f(x) := -\frac{t}{2\pi}\log x,\qquad a \le x \le b,
\]
so that 
\[
\sum_{a < n \le b}n^{-it} = S_f(a, b - a).
\]
Also, let
\[
 g_r(x) := f(x + r) - f(x) = -\frac{t}{2\pi}\log\left(1 + \frac{r}{x}\right),\qquad a \le x \le b - r,
\]
\[
\beta := g_r'(a),\qquad \alpha := g_r'(b - r).
\]
Note that since $b \le ha$, we have $\beta \le h\alpha$. Furthermore, let $x_{\nu}$ be such that $g'_r(x_{\nu}) = \nu$, i.e.
\[
x_\nu := \frac{1}{2}\sqrt{r^2 + \frac{2tr}{\pi \nu}} - \frac{r}{2}.
\]
Finally, define 
\[
\phi_r(\nu) := g_r(x_{\nu}) - \nu x_{\nu}.
\]

\subsubsection{Applying the $A^3B(0, 1)$ process}
We begin by considering the exponential sum 
\[
S_{\phi_r}(\alpha, \beta - \alpha) = \sum_{\alpha < n \le \beta}e(\phi_r(n)),
\]
which is an intermediary sum encountered prior to applying the final $AB$ process. We bound this sum using the $5$th derivative test, which corresponds to the $A^3B(0, 1)$ process. 
Via a direct computation, we have
\begin{align}
|\phi_r^{(5)}(\nu)| &= \frac{3t\sqrt{r}}{2\sqrt{\pi}\nu^{9/2}(\pi r\nu + 2t)^{7/2}}\left(8\pi^3r^3\nu^3 + 36\pi^2r^2t\nu^2 + 60\pi r \nu t^2 + 35t^3\right)\notag\\
&= \frac{3}{2\sqrt{\pi}(x + 2)^{7/2}}\left(8x^3 + 36x^2 + 60x + 35\right)\frac{(tr)^{1/2}}{\nu^{9/2}}\label{phi_5th_deriv_equality}
\end{align}
where $x := \pi r \nu / t$. For all $\alpha \le \nu \le \beta$, we have $x \in (0, 1/4]$, since 
\[
\frac{\pi r \beta}{t} = \frac{\pi r g_r'(a)}{t} = \frac{1}{2}\frac{r^2}{a(a + r)} \le \frac{1}{4}
\]
as $h \le 2$ implies $r \le a$, and $x^2/(a(a + x))$ is increasing for $x > 0$.\footnote{In fact a much sharper inequality can be applied here, since we ultimately take $r = o(a)$. However, such optimisations do not appear to affect the final result.} Therefore,
\begin{equation}
\label{xfunction_bound}
\frac{6704}{2187} \le \frac{8x^3 + 36x^2 + 60x + 35}{(x + 2)^{7/2}} \le \frac{35}{8\sqrt{2}}.
\end{equation}
Furthermore, 
\[
\frac{tr}{2\pi a^2} \ge \frac{tr}{2\pi a(a + r)} = g_r'(a) = \beta \ge \nu \ge \alpha = g_r'(b - r) = \frac{tr}{2\pi b(b - r)} \ge \frac{tr}{2\pi h^2 a^2}
\]
so that
\begin{equation}\label{trv_bound}
(2\pi)^{9/2}\frac{a^{9}}{t^{4}r^4} \le \frac{(tr)^{1/2}}{\nu^{9/2}} \le (2\pi)^{9/2}h^9 \frac{a^{9}}{t^{4}r^4}.
\end{equation}
Combining \eqref{phi_5th_deriv_equality}, \eqref{xfunction_bound} and \eqref{trv_bound}, for $\nu \in [\alpha, \beta]$, we have
\begin{equation}\label{phi_5thderiv_bounds}
\lambda_5 \le |\phi_r^{(5)}(\nu)| \le h_5\lambda_5
\end{equation}
where 
\[
\lambda_5 = \frac{53632\sqrt{2} \pi^4}{729}\frac{a^{9}}{t^{4}r^4},\qquad h_5 := \frac{76545\sqrt{2}}{107264} h^9.
\]
Meanwhile, by the mean-value theorem we have $\beta - \alpha \le (b - a - r)g''(\xi)$ for some $\xi \in [a, b - r]$, and so by directly computing $g''(x)$ we have 
\begin{equation}\label{N_bound}
\beta - \alpha < a(h - 1)\frac{tr}{2\pi}\frac{2a + r}{a^2(a + r)^2} \le \frac{h - 1}{\pi}\frac{tr}{a^2}
\end{equation}
We apply Lemma \ref{kth_deriv_test} with $k = 5$, $f = \phi_r$, $N = b - a$ and $h = h_5$ to obtain, using \eqref{phi_5thderiv_bounds} and \eqref{N_bound},
\begin{equation}\label{A3B01_exponent_pair}
\begin{split}
|S_{\phi_r}(\alpha, \beta - \alpha)| &\le A_5(\eta_1, h_5) h_5^{1/8}(\beta - \alpha)\lambda_5^{1/30} + B_5(\eta_1) (\beta - \alpha)^{7/8}\lambda_5^{-1/30}\\
&< c_1a^{-17/10}(tr)^{13/15} + c_2 a^{-41/20}(tr)^{121/120}
\end{split}
\end{equation}
for any $\eta_1 > 0$, where 
\[
c_1 = A_5(\eta_1, h_5) b_1 h^{9/8}(h - 1),\qquad c_2 = B_5(\eta_1)b_2(h - 1)^{7/8}
\]
and 
\[
b_1 := \frac{1}{\pi}\left(\frac{53632\sqrt{2}\pi^4}{729}\right)^{1/30}\left(\frac{76545\sqrt{2}}{107264}\right)^{1/8},
\]
\[
b_2 := \frac{1}{\pi^{7/8}}\left(\frac{729}{53632\sqrt{2}\pi^4}\right)^{1/30}.
\]
Note that the leading term of \eqref{A3B01_exponent_pair} corresponds to the $A^3B(0, 1) = (1/30, 13/15)$ exponent pair. 

\subsubsection{Applying the $B$ process}
Equipped with a bound for $S_{\phi_r}$, we apply the $B$ process (Lemma \ref{titchmarsh_lem_46}) with $f = g$. The end result of this subsection is an explicit $BA^3B(0, 1) = (11/30, 8/15)$ exponent pair. To do this we first require a few intermediary results. To begin, note that
\[
g_r''(x_{\nu}) \ge g_r''(b - r) = \frac{tr}{2\pi}\frac{2b - r}{b^2(b - r)^2} \ge \frac{tr}{\pi}\frac{1}{(b - r/2)^3} > \frac{tr}{\pi b^3} \ge \frac{tr}{\pi h^3 a^3}.
\]
Here, the second inequality follows from the arithmetic-geometric means inequality. Therefore, by partial summation and using \eqref{A3B01_exponent_pair},
\begin{equation}\label{partial_summation_poisson_summation}
\begin{split}
\left|\sum_{\alpha < \nu \le \beta}\frac{e(\phi_r(\nu))}{|g_r''(x_\nu)|^{1/2}}\right| &\le \pi^{1/2}h^{3/2}\frac{a^{3/2}}{(tr)^{1/2}}\max_{\alpha < L \le \beta}S_{\phi_r}(\lfloor \alpha \rfloor, L) \\
 &\le c_3a^{-1/5}(tr)^{11/30} + c_4a^{-11/20}(tr)^{61/120}
\end{split}
\end{equation}
where
\[
c_3 = b_1 \pi^{1/2} A_5 h^{21/8}(h - 1),\qquad c_4 = b_2 \pi^{1/2} B_5 h^{3/2}(h - 1)^{7/8}.
\]
Additionally, note that
\[
\frac{tr}{\pi h^3a^3} \le |g_r''(x)| =  \frac{tr(2x + r)}{2\pi x^2(x + r)^2} \le \frac{tr}{\pi a^3},
\]
and thus
\begin{equation}\label{g_2nd_deriv_test}
\lambda_2 \le |g_r''(x)| < h_2\lambda_2,\qquad \lambda_2 := \frac{tr}{\pi h^3 a^3}, \qquad h_2 := h^3.
\end{equation}
Similarly, 
\[
|g_r'''(x)| = \frac{tr}{\pi}\frac{3x^2 + 3xr + r^2}{x^3(x + r)^3} \in \left[\frac{3tr}{\pi h^4 a^4}, \frac{3tr}{\pi a^4}\right],
\]
\begin{equation}\label{g_3rd_deriv_test}
\lambda_3 \le |g_r'''(x)| \le h_3\lambda_3,\qquad \lambda_3 := \frac{3tr}{\pi h^4 a^4},\qquad h_3 := h^4.
\end{equation}
Applying Lemma \ref{titchmarsh_lem_46} and using \eqref{partial_summation_poisson_summation}, \eqref{g_2nd_deriv_test} and \eqref{g_3rd_deriv_test}, we finally obtain
\begin{equation}\label{S_g_abr_bound}
|S_{g_r}(a, b - r - a)| \le c_3a^{-1/5}(tr)^{11/30} + c_4a^{-11/20}(tr)^{61/120} + E
\end{equation}
where, from Lemma \ref{imp_poisson_summation_formula}, the error term $E$ satisfies
\begin{align*}
|E| &\le 4.686\lambda_2^{-1/2} + \frac{5}{\pi}\log(\beta - \alpha + 2) + \frac{2\cdot 3^{2/3}}{\pi^{2/3}} h_2h_3^{1/3}(b - a)\lambda_3^{1/3} + 6.
\end{align*}
Setting 
\begin{align*}
   T_1 &=4.686\lambda_2^{-1/2},\qquad  T_2=\frac{5}{\pi}\log(\beta - \alpha + 2),\\
   T_3&=\frac{2\cdot 3^{2/3}}{\pi^{2/3}} h_2h_3^{1/3}(b - a)\lambda_3^{1/3},\qquad T_4=6, 
\end{align*}
and substituting \eqref{g_2nd_deriv_test} and \eqref{g_3rd_deriv_test}, we have
\begin{equation}\label{T1T2_error_bound}
T_1 = 4.686\sqrt{\pi}h^{3/2}\frac{a^{3/2}}{(tr)^{1/2}},\qquad T_3 \le \frac{6}{\pi}h^{3}(h - 1)\frac{(tr)^{1/3}}{a^{1/3}}.
\end{equation}
Furthermore, since $\log(2 + x) \le \log 2 + 0.893x^{1/3}$ for all $x > 0$, we have, using \eqref{N_bound}, 
\[
\log(2 + \beta - \alpha) \le \log 2 + 0.893(\beta - \alpha)^{1/3} \le \log 2 + \frac{0.893}{a^{1/3}}\left(\frac{h - 1}{\pi}\frac{tr}{a}\right)^{1/3}.
\]
This implies, from $a \ge \theta_2t^{7/17} \ge \theta_2 t_0^{7/17}$, that
\begin{equation}\label{T3_error_bound}
T_2 \le \frac{5}{\pi}\left(\log 2 + \frac{0.893}{\theta_2^{1/3}t_0^{7/51}}\left(\frac{h - 1}{\pi}\right)^{1/3}\frac{(tr)^{1/3}}{a^{1/3}}\right).
\end{equation}
Combining \eqref{S_g_abr_bound}, \eqref{T1T2_error_bound} and \eqref{T3_error_bound}, we have 
\begin{equation}\label{BA3B_process}
\begin{split}
|S_{g_r}(a, b - r - a)| &\le c_3a^{-1/5}(tr)^{11/30} + c_4a^{-11/20}(tr)^{61/120} \\
&\qquad\qquad\qquad + E_4\frac{a^{3/2}}{(tr)^{1/2}} + E_5\frac{(tr)^{1/3}}{a^{1/3}} + E_6,
\end{split}
\end{equation}
where 
\[
E_4 := 4.686\sqrt{\pi}h^{3/2},\qquad E_5 := \frac{0.893}{\theta_2^{1/3}t_0^{7/51}}\frac{5}{\pi}\left(\frac{h - 1}{\pi}\right)^{1/3} + \frac{6}{\pi}h^{3}(h - 1),
\]
\[
E_6 := 6 + \frac{5}{\pi}\log 2.
\]

\subsubsection{Applying the $A$ process} To complete the proof we apply the $A$ process once more to obtain the exponent pair $ABA^3B(0, 1) = (11/82, 57/82)$. To do so we rely on the following inequality, which can be found in e.g.\ \cite{patel_explicit_2021}:
\begin{equation}\label{summation_formula_patel}
\sum_{r = 1}^q\left(1 - \frac{r}{q}\right)r^{s} \le \frac{q^{1 + s}}{(1 + s)(2 + s)},\qquad -1 < s \le 1,
\end{equation}
for all integers $q \ge 1$. Applying this formula, and using \eqref{BA3B_process}, we obtain
\begin{align*}
\frac{2}{q}\sum_{r = 1}^{q - 1}\left(1 - \frac{r}{q}\right)\left|S_{g_r}(a, b - r - a)\right| &\le \frac{1800}{2911}c_3 a^{-1/5}(qt)^{11/30} + \frac{28800}{54481}c_4 a^{-11/20}(qt)^{61/120} \\
&\qquad\qquad + \frac{8}{3}E_4\frac{a^{3/2}}{(tq)^{1/2}} + \frac{9}{14}E_5\frac{(tq)^{1/3}}{a^{1/3}} + E_6.
\end{align*}
We use this in Lemma \ref{weyl_differencing}, together with $b - a \le (h - 1)a$, to obtain
\begin{equation}
\begin{split}
\left|S_f(a, b - a)\right|^2 &\le \left(h - 1 + \frac{q}{a}\right)\Bigg(\frac{(h - 1)a^2}{q} + \frac{1800}{2911}c_3 a^{4/5}(qt)^{11/30} \\ 
&\quad + \frac{28800}{54481}c_4 a^{9/20}(qt)^{61/120} + \frac{8}{3}E_4\frac{a^{5/2}}{(tq)^{1/2}} + \frac{9}{14}E_5a^{2/3}(tq)^{1/3} + E_6a\Bigg).
\end{split}
\end{equation}
We choose $q = \lfloor \theta_1 a^{36/41}t^{-11/41}\rfloor$ for some $\theta_1 > 0$ to be chosen later, so that
\[
\theta_1 a^{36/41}t^{-11/41} - 1 \le q \le \theta_1 a^{36/41}t^{-11/41}.
\]
Now, with $q_0$ defined in the lemma statement, we have
\[
q_0 = \theta_1\theta_2^{7/17}t_0^{65/697} \le \theta_1 a^{36/41}t^{-11/41} \le q + 1
\]
and hence by assumption, $q \ge 1$. Observe that the following inequalities hold
\[
\frac{a^2}{q} \le \frac{a^2}{\theta_1 a^{36/41}t^{-11/41} - 1} \le \frac{\theta_1^{-1}}{1 - q_0^{-1}}\cdot a^{46/41}t^{11/41},
\]
\[
\frac{q}{a} \le \frac{\theta_1 a^{36/41}t^{-11/41}}{a} \le \frac{\theta_1}{(\theta_2 t^{7/17})^{5/41}t^{11/41}} \le \frac{\theta_1}{\theta_2^{5/41}t_0^{222/697}},
\]
\begin{align*}
\frac{a^{5/2}}{(qt)^{1/2}} \le \frac{1}{(\theta_1a^{36/41}t^{-11/41} - 1)^{1/2}}a^{5/2}t^{-1/2} &\le \left(\theta_1 - a^{-36/41}t^{11/41}\right)^{-1/2}a^{169/82}t^{-15/41} \\
&\le \left(\theta_1 - \frac{1}{\theta_2t_0^{100/697}}\right)^{-1/2}a^{169/82}t^{-15/41},
\end{align*}
\[
a^{2/3}(tq)^{1/3} \le \theta_1^{1/3}a^{118/123}t^{10/41}.
\]
Using the above inequalities, we obtain
\begin{align*}
&|S_f(a, b - a)|^2 \le \left(h - 1 + \frac{\theta_1}{\theta_2^{5/41}t_0^{222/697}}\right)\bigg(\left(\frac{\theta_1^{-1}(h - 1)}{1 - q_0^{-1}} + \frac{1800}{2911}\theta_1^{11/30}c_3\right)a^{46/41}t^{11/41} \\
&+ \left(\frac{28800}{54481}\theta_1^{61/120}c_4\right)a^{147/164}t^{61/164} + \frac{8}{3}E_4\left(\theta_1 - \frac{1}{\theta_2t_0^{100/697}}\right)^{-1/2}a^{169/82}t^{-15/41} \\
&+ \frac{9}{14}\theta_1^{1/3}a^{118/123}t^{10/41} + E_6a\bigg)
\end{align*}
Taking square roots of both sides, applying $\sqrt{x_1 + \cdots + x_n} \le \sqrt{x_1} + \cdots + \sqrt{x_n}$, and substituting the values of $c_3$, $c_4$, $E_4$, $E_5$ and $E_6$, the desired result follows. 
\end{proof}

\begin{lemma}\label{S3_bound_lem}
Let $100 \le t_0 \le t \le t_1$, $h > 1$ and $\theta_1, \theta_2, \theta_3, \eta_1 > 0$ be constants jointly satisfying the conditions of Lemma \ref{exponential_sum_lemma}. Furthermore assume that $h_0 := h/(1 - \theta_1 t_0^{-7/17}) \in (1, 2]$. Then
\[
|S_3| \le C_4(t_0, t_1, h, \eta_1, \theta_1, \theta_2, \theta_3) t^{27/164}
\]
where
\begin{equation}\label{S_3_final_bound_finite_interval}
\begin{split}
C_4 &:= C_1(h_0) \mu_1\left(\frac{5}{82}\right) + C_2(h_0) \mu_2\left(\frac{17}{328}\right)h^{17K(t_1)/328}t_0^{-3/656}\\
& + E_1(h_0)\mu_1\left(\frac{87}{164}\right)t_0^{-27/328} + E_2(h_0)\mu_2\left(\frac{5}{246}\right)h^{5K(t_1)/246}t_0^{-13/246} + E_3(h_0)K(t_1)t_0^{-27/164},
\end{split}
\end{equation}
\begin{equation}
\mu_1(\alpha) := \frac{1}{(2\pi)^{\alpha/2}}\frac{1 - h^{-\alpha K(t_1)}}{h^{\alpha} - 1}, \qquad \mu_2(\alpha) := \mu_1(\alpha)\left(1 - \frac{h}{\theta_2t_0^{7/17}}\right)^{-\alpha}
\end{equation}
and $C_1$, $C_2$, $E_1$, $E_2$ and $E_3$ are as defined in Lemma \ref{exponential_sum_lemma}. Furthermore, for all $t \ge t_0$, 
\[
|S_3| \le C_5(t_0, h, \eta_1, \theta_1, \theta_2, \theta_3) t^{27/164}
\]
where 
\begin{equation}\label{S_3_final_bound_infinite_interval}
\begin{split}
C_5 &:= C_1(h_0)\mu_3\left(\frac{5}{82}\right) + C_2(h_0)\mu_4\left(\frac{17}{328}\right) + E_1(h_0)\mu_3\left(\frac{87}{164}\right)t_0^{-27/328}\\
&\quad
+ E_2(h_0)\mu_4\left(\frac{5}{246}\right)t_0^{-427/8364} + E_3(h_0)\left(\frac{\frac{3}{34}\log t_0 - \log(\theta_2\sqrt{2\pi})}{\log h} + 1\right)t_0^{-27/164},
\end{split}
\end{equation}
\begin{equation}
\mu_3(\alpha) := \frac{1}{(2\pi)^{\alpha/2}(h^{\alpha} - 1)}, \qquad \mu_4(\alpha) := \left(1 - \frac{h}{\theta_2t_0^{7/17}}\right)^{-\alpha}\frac{\left(h/\theta_2\right)^{\alpha} - \sqrt{2\pi}t_0^{-3/34}}{1 - h^{-\alpha}}.
\end{equation}
\end{lemma}
\begin{proof}
With $a_k$ as defined in \eqref{a_k_defn}, we have 
\[
\frac{h^{-(k - 1)}\sqrt{\frac{t}{2\pi}}}{h^{-k}\sqrt{\frac{t}{2\pi}} - 1} \le h\left(\frac{1}{1 - a_K^{-1}}\right) < h_0,
\]
say. Therefore, we may apply Lemma \ref{exponential_sum_lemma}, with $h = h_0$, $a = a_k$, $b = a_{k - 1}$ and $t_0 \le t$ to obtain, via partial summation,
\begin{align*}
\left|\sum_{a_{k} < n \le a_{k - 1}}n^{-1/2 - it}\right| &\le a_k^{-1/2}\max_{a_k < L \le a_{k - 1}}S_f(a_k, L - a_k) \\
&\le C_1 a^{5/82}t^{11/82} + C_2 a^{-17/328}t^{61/328} + E_1a^{87/164}t^{-15/82} \\
&\qquad\qquad + E_2a^{-5/246}t^{5/41} + E_3,
\end{align*}
and thus 
\begin{equation}\label{S3_initial_bound}
\begin{split}
|S_3| &\le \sum_{k = 1}^{K}\left|\sum_{a_k < n \le a_{k - 1}}n^{-1/2 - it}\right| \le C_1 t^{11/82}\sum_{k = 1}^Ka_k^{5/82} + C_2t^{61/328}\sum_{k = 1}^{K}a_k^{-17/328}\\
&\qquad\qquad\qquad\qquad + E_1t^{-15/82}\sum_{k = 1}^{K}a_k^{87/164} + E_2t^{5/41}\sum_{k = 1}^{K}a_k^{-5/246} + E_3K.
\end{split}
\end{equation}
Since 
\[
h^{-k} \sqrt{\frac{t}{2\pi}} - 1 \le a_k \le h^{-k} \sqrt{\frac{t}{2\pi}}
\]
we have, for any $\alpha > 0$ and $t_0 \le t \le t_1$,
\begin{equation}\label{ak_pos_exp_upper_bound_finite_interval}
\sum_{k = 1}^Ka_k^{\alpha} \le \sum_{k = 1}^K\left(h^{-k}\sqrt{\frac{t}{2\pi}}\right)^{\alpha} = \left(\frac{t}{2\pi}\right)^{\alpha/2}\sum_{k = 1}^K(h^{-\alpha})^k = \mu_1(\alpha)t^{\alpha/2}
\end{equation}
and, again for $\alpha > 0$ and $t_0 \le t \le t_1$,
\begin{align}
\sum_{k = 1}^Ka_k^{-\alpha} &\le \sum_{k = 1}^K\left(h^{-k}\sqrt{\frac{t}{2\pi}} - 1\right)^{-\alpha} = \left(1 - h^k\sqrt{\frac{2\pi}{t}}\right)^{-\alpha}\sum_{k = 1}^K\left(h^{-k}\sqrt{\frac{t}{2\pi}}\right)^{-\alpha} \notag\\
&< \left(1 - \frac{h}{\theta_2t_0^{7/17}}\right)^{-\alpha}\left({\frac{t}{2\pi}}\right)^{-\alpha/2}\frac{h^{\alpha K} - 1}{1 - h^{-\alpha}} = \mu_2(\alpha)h^{\alpha K(t_1)}t^{-\alpha/2}. \label{ak_neg_exp_upper_bound_finite_interval}
\end{align}
Substituting these into \eqref{S3_initial_bound}, we obtain the estimate
\begin{align*}
|S_3| &\le C_1\mu_1\left(\frac{5}{82}\right)t^{27/164} + C_2\mu_2\left(\frac{17}{328}\right)h^{17K(t_1)/328}t^{105/656}\\
&\qquad\qquad + E_1\mu_1\left(\frac{87}{164}\right)t^{27/328} + E_2\mu_2\left(\frac{5}{246}\right)h^{5K(t_1)/246}t^{55/492} + E_3K(t_1),
\end{align*}
which gives \eqref{S_3_final_bound_finite_interval} from $t \ge t_0$, and forms the main bound for $S_3$ for $t$ in finite intervals $[t_0, t_1]$. To obtain a bound holding for all $[t_0, \infty)$, we use 
\[
h^{K} < h^{\left(\frac{3}{34}\log t - \log(\theta_2\sqrt{2\pi})\right)/\log h + 1} = \frac{h}{\theta_2\sqrt{2\pi}}t^{3/34}
\]
to continue the argument from \eqref{ak_pos_exp_upper_bound_finite_interval} and \eqref{ak_neg_exp_upper_bound_finite_interval} to obtain, for $t \ge t_0$,
\[
\sum_{k = 1}^Ka_k^{\alpha} < \mu_3(\alpha),
\]
\[
\sum_{k = 1}^Ka_k^{-\alpha} < \left(1 - \frac{h}{\theta_2t_0^{7/17}}\right)^{-\alpha}\frac{\left(h/\theta_2\right)^{\alpha} - \sqrt{2\pi}t_0^{-3/34}}{1 - h^{-\alpha}}t^{-7\alpha/17} = \mu_4(\alpha)t^{-7\alpha/17}.
\]
Equation \eqref{S_3_final_bound_infinite_interval} then follows from substituting these estimates into \eqref{S3_initial_bound} and using $t \ge t_0$.
\end{proof}

\subsection{Computations}\label{sec:computation}

For each row $(\log t_0, \log t_1, h_1, h_2, \eta_1, \eta_2, \theta_1, \theta_2, \theta_3, A)$ of Table \ref{coefficientstable}, we substitute the relevant parameter values into Lemma \ref{S1_bound_lem}, \ref{region_2_bound} and \ref{S3_bound_lem} to verify, in each case, that 
\[
|\zeta(1/2 + it)| \le At^{27/164},\qquad t_0 \le t \le t_1. 
\]
Upon inspection, we have $A \le 66.7$ in each case, which proves Theorem \ref{main_thm} for $\exp(60) \le t \le \exp(875)$. These parameters are found via a stochastic optimisation routine so are not necessarily globally optimal, however they suffice for justifying an upper bound on the constant factor in Theorem \ref{main_thm}. 

In addition, by taking $t_0 = \exp(875)$, $\theta_1 = 1.14283$, $\theta_2 = 261658$, $\theta_3 = 2.53087\cdot 10^{-11}$, $h_1 = 1.01563$, $h_2 = 1.00270$, $\eta_1 = 1.59875$ and $\eta_2 = 0.828895$ in Lemma \ref{S1_bound_lem}, \ref{region_2_bound} and \eqref{S_3_final_bound_infinite_interval}, we obtain 
\[
|\zeta(1/2 + it)| \le 66.7 t^{27/164},\qquad t \ge \exp(875).
\]
Note that in the application of Lemma \ref{region_2_bound}, we take the limit as $t_1 \to\infty$. This implies Theorem \ref{main_thm} for $t \ge \exp(875)$.

For small values of $t$, we use the following bound
\begin{equation}\label{hpy_intermediate_bound}
|\zeta(1/2 + it)| \le 0.478013 t^{1/6}\log t + 3.853165 t^{1/6} - 2.914229,\qquad t \ge 10^{12},
\end{equation}
proved in \S3.4 of \cite{hiary_improved_2022}. This estimate covers the range $10^{12} \le t < \exp(60)$. Finally, for $3 \le t < 10^{12}$, we use the classical van der Corput estimate \eqref{hpy_estimate}. This completes the proof of Theorem \ref{main_thm}. Lastly, we note that the bounds in Table \ref{coefficientstable} improve on both \eqref{hpy_intermediate_bound} and \eqref{patel_subweyl} for all $t \ge \exp(60.6)$, and is thus the sharpest known bound on $\zeta(1/2 + it)$ in this range. 

\section{Conclusion and future work}
Theorem \ref{main_thm} represents the first of many successively sharper sub-Weyl bounds of the form $\zeta(1/2 + it) \ll_{\epsilon} t^{\theta + \epsilon}$ obtainable from van der Corput's method. The next few values of $\theta$, due to \cite{phillips_zetafunction_1933}, \cite{titchmarsh_order_1942}, \cite{min_order_1949}, \cite{Haneke_verscharfung_1963} and \cite{kolesnik_order_1982} respectively, are 
\[
\frac{229}{1392},\quad \frac{19}{116},\quad \frac{15}{92},\quad \frac{6}{37},\quad \frac{35}{216}.
\]
The first result, $\theta = 229/1392$, can be obtained via the exponent pair 
\[
ABA^3BA^2BA^2B(0, 1) = \left(\frac{97}{696}, \frac{20}{29}\right)
\]
and can thus be made explicit using a similar but longer version of the arguments presented in this paper. Exponents starting from $\theta = 19/116$, however, rely on estimates of higher-dimensional exponential sums. For example, in the two-dimensional case, the function $g_r(x) = f(x + r) - f(x)$ encountered in the $A$ process (Lemma \ref{weyl_differencing}) can be treated as a function of two variables, $r$ and $x$. Such a sum can be estimated using two-dimensional analogs of the $A$ and $B$ processes.

The main obstacles to computing an explicit version of such results are difficulties with the two-dimensional Poisson summation formula. In the two-dimensional analog of the $B$ process, the factor $|f''(x_{\nu})|^{-1/2}$ appearing in Lemma \ref{imp_poisson_summation_formula} is replaced by the Hessian of $f$, defined by 
\[
Hf(x, y) := \det \begin{bmatrix}
\partial_{xx}f & \partial_{xy}f\\
\partial_{yx}f & \partial_{yy}f
\end{bmatrix}.
\]
However, if $Hf$ vanishes within the rectangle of summation, as is the case when bounding $\zeta(1/2 + it)$, it can be difficult to control the transformed sum. Successful implementations \cite{titchmarsh_lattice_1935, min_order_1949} of two-dimensional exponent pairs rely on elaborate arguments to isolate problematic regions within the summation rectangle, and applying the trivial bound in those regions instead. Explicit versions of higher-dimensional Poisson summation formulae will be investigated in a future article. 

\section*{Acknowledgements}
We would like to thank Timothy S. Trudgian and Ghaith A. Hiary for their continuous support and helpful suggestions throughout the writing of this paper. 

\newpage

\begin{center}
\begin{longtable}{|c|c|c|c|c|c|c|c|c|c|}
\caption{Parameter values used in the proof of Theorem \ref{main_thm}.}
\label{coefficientstable}
\\
\hline
$\log t_0$ & $\log t_1$ & $h_1$ & $h_2$ & $\eta_1$ & $\eta_2$ & $\theta_1$ & $\theta_2$ & $\theta_3$ & $A$\\
\hline
$60$ & $65$ & $1.02932$ & $1.06726$ & $1.72183$ & $1.02275$ & $0.957426$ & $0.180062$ & $0.172999$ & $37.10$\\
\hline
$65$ & $70$ & $1.02739$ & $1.06227$ & $1.65356$ & $1.08576$ & $0.938332$ & $0.189995$ & $0.121681$ & $38.09$\\
\hline
$70$ & $75$ & $1.02561$ & $1.05961$ & $1.76719$ & $1.03304$ & $0.926944$ & $0.194964$ & $0.0811119$ & $39.05$\\
\hline
$75$ & $80$ & $1.0231$ & $1.05617$ & $1.94747$ & $1.04071$ & $0.927416$ & $0.20589$ & $0.0549285$ & $39.98$\\
\hline
$80$ & $85$ & $1.02302$ & $1.05429$ & $2.20976$ & $1.03294$ & $0.915828$ & $0.223422$ & $0.0383106$ & $40.88$\\
\hline
$85$ & $90$ & $1.02248$ & $1.05174$ & $1.98074$ & $1.00763$ & $0.917169$ & $0.230051$ & $0.0265495$ & $41.75$\\
\hline
$90$ & $95$ & $1.02221$ & $1.04803$ & $1.94585$ & $1.02266$ & $0.908887$ & $0.254651$ & $0.0203407$ & $42.60$\\
\hline
$95$ & $100$ & $1.02211$ & $1.04783$ & $2.12321$ & $1.04079$ & $0.93665$ & $0.290334$ & $0.0137463$ & $43.41$\\
\hline
$100$ & $105$ & $1.02165$ & $1.04598$ & $1.93655$ & $1.06102$ & $0.906136$ & $0.305624$ & $0.00865144$ & $44.19$\\
\hline
$105$ & $110$ & $1.02191$ & $1.0438$ & $2.03193$ & $0.948977$ & $0.920571$ & $0.320444$ & $0.0076111$ & $44.96$\\
\hline
$110$ & $115$ & $1.02142$ & $1.04532$ & $1.98343$ & $1.01929$ & $0.931888$ & $0.364937$ & $0.00567012$ & $45.70$\\
\hline
$115$ & $120$ & $1.02014$ & $1.04372$ & $2.08479$ & $0.988021$ & $0.931926$ & $0.391814$ & $0.00426777$ & $46.41$\\
\hline
$120$ & $125$ & $1.02155$ & $1.03987$ & $2.02632$ & $0.972639$ & $0.932056$ & $0.427062$ & $0.00308127$ & $47.10$\\
\hline
$125$ & $130$ & $1.02124$ & $1.0422$ & $1.90876$ & $1.02384$ & $0.919159$ & $0.44983$ & $0.00146827$ & $47.76$\\
\hline
$130$ & $135$ & $1.02096$ & $1.03965$ & $1.93958$ & $1.0108$ & $0.926399$ & $0.502711$ & $0.00196936$ & $48.42$\\
\hline
$135$ & $140$ & $1.02133$ & $1.03736$ & $2.0425$ & $0.9705$ & $0.935329$ & $0.549766$ & $0.00129695$ & $49.05$\\
\hline
$140$ & $145$ & $1.02187$ & $1.03816$ & $1.90804$ & $1.00432$ & $0.931849$ & $0.592399$ & $0.000924586$ & $49.66$\\
\hline
$145$ & $150$ & $1.02183$ & $1.03717$ & $1.99931$ & $0.968481$ & $0.934159$ & $0.652989$ & $0.00104845$ & $50.26$\\
\hline
$150$ & $155$ & $1.02183$ & $1.03826$ & $1.88929$ & $0.950307$ & $0.943886$ & $0.718618$ & $0.000853561$ & $50.84$\\
\hline
$155$ & $160$ & $1.02048$ & $1.03617$ & $1.93121$ & $0.98237$ & $0.938295$ & $0.738158$ & $0.000619765$ & $51.39$\\
\hline
$160$ & $165$ & $1.02138$ & $1.03453$ & $1.96189$ & $0.977865$ & $0.935631$ & $0.857581$ & $0.00064869$ & $51.94$\\
\hline
$165$ & $170$ & $1.02094$ & $1.03251$ & $1.89843$ & $0.988454$ & $0.940228$ & $0.909567$ & $0.000521171$ & $52.46$\\
\hline
$170$ & $175$ & $1.02158$ & $1.03544$ & $1.89217$ & $0.98871$ & $0.950508$ & $1.00477$ & $0.000414171$ & $52.97$\\
\hline
$175$ & $180$ & $1.0214$ & $1.03158$ & $1.93452$ & $1.04548$ & $0.939419$ & $1.03725$ & $0.000314426$ & $53.47$\\
\hline
$180$ & $185$ & $1.02039$ & $1.03179$ & $2.03717$ & $0.972755$ & $0.95516$ & $1.18322$ & $0.000517973$ & $53.97$\\
\hline
$185$ & $190$ & $1.02096$ & $1.03431$ & $1.91394$ & $0.982296$ & $0.942759$ & $1.28455$ & $0.000322386$ & $54.43$\\
\hline
$190$ & $195$ & $1.02108$ & $1.02999$ & $1.92547$ & $0.97085$ & $0.954305$ & $1.36094$ & $0.000475754$ & $54.90$\\
\hline
$195$ & $200$ & $1.02121$ & $1.03205$ & $1.9824$ & $0.949768$ & $0.944675$ & $1.51919$ & $0.000385305$ & $55.33$\\
\hline
$200$ & $205$ & $1.02125$ & $1.02682$ & $1.95701$ & $0.982448$ & $0.947591$ & $1.63942$ & $0.000229165$ & $55.75$\\
\hline
$205$ & $210$ & $1.0212$ & $1.02869$ & $1.89377$ & $0.953281$ & $0.952798$ & $1.70281$ & $0.000313579$ & $56.18$\\
\hline
$210$ & $215$ & $1.0209$ & $1.03072$ & $1.99679$ & $0.979303$ & $0.953546$ & $1.92024$ & $0.000228668$ & $56.58$\\
\hline
$215$ & $220$ & $1.02096$ & $1.02656$ & $1.95651$ & $0.928927$ & $0.967107$ & $2.09168$ & $0.000336118$ & $56.98$\\
\hline
$220$ & $225$ & $1.02094$ & $1.02878$ & $1.89319$ & $0.939551$ & $0.954351$ & $2.33168$ & $0.000284262$ & $57.36$\\
\hline
$225$ & $230$ & $1.02105$ & $1.02806$ & $2.06086$ & $0.965398$ & $0.954194$ & $2.5612$ & $0.000398031$ & $57.74$\\
\hline
$230$ & $235$ & $1.01954$ & $1.02547$ & $2.04008$ & $0.929068$ & $0.954843$ & $2.67558$ & $0.000256528$ & $58.09$\\
\hline
$235$ & $240$ & $1.02038$ & $1.02633$ & $1.8313$ & $0.941395$ & $0.963688$ & $3.04059$ & $0.000209508$ & $58.43$\\
\hline
$240$ & $245$ & $1.02026$ & $1.02293$ & $1.92062$ & $0.939602$ & $0.967721$ & $3.18847$ & $0.000202481$ & $58.77$\\
\hline
$245$ & $250$ & $1.02015$ & $1.02567$ & $1.91183$ & $0.908205$ & $0.973884$ & $3.49943$ & $0.000335679$ & $59.11$\\
\hline
$250$ & $255$ & $1.02167$ & $1.02194$ & $1.82234$ & $0.926453$ & $0.961316$ & $3.67918$ & $0.000399662$ & $59.44$\\
\hline
$255$ & $260$ & $1.0207$ & $1.02546$ & $1.85731$ & $0.981673$ & $0.972605$ & $4.17292$ & $0.000240777$ & $59.73$\\
\hline
$260$ & $265$ & $1.02011$ & $1.02388$ & $2.09485$ & $0.969026$ & $0.965687$ & $4.65096$ & $0.000192086$ & $60.02$\\
\hline
$265$ & $270$ & $1.02066$ & $1.02359$ & $1.96056$ & $0.893437$ & $0.969203$ & $4.83374$ & $0.000251274$ & $60.32$\\
\hline
$270$ & $275$ & $1.0201$ & $1.02217$ & $1.83113$ & $0.987098$ & $0.964569$ & $5.25223$ & $0.000238956$ & $60.60$\\
\hline
$275$ & $280$ & $1.02004$ & $1.02147$ & $1.8859$ & $0.958011$ & $0.982108$ & $5.73906$ & $0.000261819$ & $60.87$\\
\hline
$280$ & $285$ & $1.02017$ & $1.02028$ & $1.82159$ & $0.976256$ & $0.96597$ & $5.91678$ & $0.000223948$ & $61.13$\\
\hline
$285$ & $290$ & $1.02029$ & $1.02019$ & $1.86161$ & $0.930287$ & $0.977082$ & $6.75422$ & $0.000270409$ & $61.39$\\
\hline
$290$ & $295$ & $1.02025$ & $1.0215$ & $1.8654$ & $0.93816$ & $0.976725$ & $7.47305$ & $0.000230413$ & $61.63$\\
\hline
$295$ & $300$ & $1.02004$ & $1.02012$ & $1.80343$ & $0.90056$ & $0.981492$ & $7.99716$ & $0.000137621$ & $61.86$\\
\hline
$300$ & $305$ & $1.01958$ & $1.02158$ & $1.7938$ & $0.917443$ & $0.980441$ & $9.15625$ & $0.000194688$ & $62.10$\\
\hline
$305$ & $310$ & $1.02009$ & $1.01906$ & $1.82009$ & $0.962434$ & $0.981321$ & $9.66325$ & $0.000221732$ & $62.32$\\
\hline
$310$ & $315$ & $1.02048$ & $1.02047$ & $1.80586$ & $0.925325$ & $0.984564$ & $10.3824$ & $0.000170812$ & $62.53$\\
\hline
$315$ & $320$ & $1.0198$ & $1.02051$ & $2.01424$ & $0.912211$ & $0.987193$ & $11.9473$ & $0.000322156$ & $62.76$\\
\hline
$320$ & $325$ & $1.02049$ & $1.01903$ & $1.89547$ & $0.935573$ & $0.983211$ & $12.4283$ & $0.000243574$ & $62.95$\\
\hline
$325$ & $330$ & $1.02015$ & $1.01867$ & $1.96936$ & $0.920717$ & $0.985354$ & $13.1599$ & $0.000243177$ & $63.14$\\
\hline
$330$ & $335$ & $1.01998$ & $1.01839$ & $1.75777$ & $0.923855$ & $0.989922$ & $14.3557$ & $0.000337639$ & $63.34$\\
\hline
$335$ & $340$ & $1.0198$ & $1.01668$ & $1.95299$ & $0.922568$ & $0.997515$ & $16.5761$ & $0.000132965$ & $63.49$\\
\hline
$340$ & $345$ & $1.0202$ & $1.01716$ & $1.81734$ & $0.94091$ & $0.998342$ & $17.9692$ & $0.000255942$ & $63.68$\\
\hline
$345$ & $350$ & $1.01961$ & $1.01542$ & $1.79264$ & $0.964049$ & $0.992996$ & $19.85$ & $0.00022936$ & $63.84$\\
\hline
$350$ & $355$ & $1.01963$ & $1.0165$ & $1.86902$ & $0.928483$ & $0.996867$ & $20.2785$ & $0.000225968$ & $64.00$\\
\hline
$355$ & $360$ & $1.02034$ & $1.01579$ & $1.91329$ & $0.935197$ & $0.999189$ & $23.4266$ & $0.000160669$ & $64.15$\\
\hline
$360$ & $365$ & $1.01909$ & $1.01585$ & $1.82083$ & $0.890385$ & $1.00133$ & $24.8655$ & $0.00020917$ & $64.30$\\
\hline
$365$ & $370$ & $1.01971$ & $1.01676$ & $1.77505$ & $0.936947$ & $1.00886$ & $27.9853$ & $0.000273506$ & $64.45$\\
\hline
$370$ & $375$ & $1.02008$ & $1.01541$ & $1.80961$ & $0.898$ & $1.01433$ & $30.7508$ & $0.000247706$ & $64.59$\\
\hline
$375$ & $380$ & $1.0194$ & $1.01554$ & $1.96358$ & $0.895426$ & $1.0133$ & $31.265$ & $0.000248272$ & $64.72$\\
\hline
$380$ & $385$ & $1.01983$ & $1.01418$ & $1.79329$ & $0.932697$ & $1.00831$ & $34.8777$ & $0.000273313$ & $64.85$\\
\hline
$385$ & $390$ & $1.01977$ & $1.01589$ & $1.84736$ & $0.91237$ & $1.01851$ & $38.5851$ & $0.000242406$ & $64.96$\\
\hline
$390$ & $395$ & $1.02013$ & $1.01406$ & $1.77664$ & $0.934046$ & $1.00421$ & $42.763$ & $0.000194332$ & $65.07$\\
\hline
$395$ & $400$ & $1.01939$ & $1.01407$ & $1.82445$ & $0.903297$ & $1.01028$ & $46.6084$ & $0.000318785$ & $65.20$\\
\hline
$400$ & $405$ & $1.01904$ & $1.01305$ & $1.825$ & $0.899607$ & $1.01356$ & $52.1417$ & $0.000220817$ & $65.29$\\
\hline
$405$ & $410$ & $1.0189$ & $1.01201$ & $1.88238$ & $0.912005$ & $1.01941$ & $53.1663$ & $0.000211265$ & $65.39$\\
\hline
$410$ & $415$ & $1.01878$ & $1.01242$ & $1.84523$ & $0.934599$ & $1.01315$ & $61.3463$ & $0.000231803$ & $65.49$\\
\hline
$415$ & $420$ & $1.01899$ & $1.01342$ & $1.79699$ & $0.898858$ & $1.02272$ & $62.7726$ & $0.000243697$ & $65.59$\\
\hline
$420$ & $425$ & $1.01885$ & $1.01242$ & $1.75512$ & $0.900257$ & $1.02809$ & $70.4483$ & $0.000182461$ & $65.66$\\
\hline
$425$ & $430$ & $1.01924$ & $1.01234$ & $1.7811$ & $0.940522$ & $1.01604$ & $73.6136$ & $0.000204745$ & $65.75$\\
\hline
$430$ & $435$ & $1.01942$ & $1.01163$ & $1.75987$ & $0.887557$ & $1.01677$ & $85.4154$ & $0.000200687$ & $65.83$\\
\hline
$435$ & $440$ & $1.01918$ & $1.01177$ & $1.77652$ & $0.85801$ & $1.03222$ & $94.0339$ & $0.000137152$ & $65.89$\\
\hline
$440$ & $445$ & $1.01931$ & $1.01146$ & $1.80958$ & $0.877476$ & $1.0233$ & $99.6689$ & $0.000145403$ & $65.97$\\
\hline
$445$ & $450$ & $1.0189$ & $1.01039$ & $1.82757$ & $0.84405$ & $1.02366$ & $111.324$ & $0.000187234$ & $66.04$\\
\hline
$450$ & $455$ & $1.01814$ & $1.01226$ & $1.6686$ & $0.939848$ & $1.03289$ & $116.834$ & $0.0001905$ & $66.10$\\
\hline
$455$ & $460$ & $1.01958$ & $1.01329$ & $1.7525$ & $0.850712$ & $1.03058$ & $130.957$ & $0.000269602$ & $66.17$\\
\hline
$460$ & $465$ & $1.01852$ & $1.0109$ & $1.796$ & $0.903805$ & $1.02878$ & $139.055$ & $0.000208527$ & $66.22$\\
\hline
$465$ & $470$ & $1.01893$ & $1.0108$ & $1.75859$ & $0.894058$ & $1.03424$ & $151.974$ & $0.000475389$ & $66.30$\\
\hline
$470$ & $475$ & $1.01912$ & $1.01125$ & $1.71125$ & $0.853318$ & $1.04161$ & $164.259$ & $0.000282763$ & $66.33$\\
\hline
$475$ & $480$ & $1.01916$ & $1.00968$ & $1.77909$ & $0.89025$ & $1.04003$ & $187.133$ & $0.00029613$ & $66.38$\\
\hline
$480$ & $485$ & $1.0178$ & $1.01075$ & $1.80647$ & $0.898548$ & $1.04088$ & $210.376$ & $0.00016592$ & $66.40$\\
\hline
$485$ & $490$ & $1.01826$ & $1.00939$ & $1.74844$ & $0.900994$ & $1.04706$ & $217.358$ & $0.000288872$ & $66.46$\\
\hline
$490$ & $495$ & $1.01802$ & $1.0088$ & $1.7787$ & $0.87362$ & $1.05188$ & $243.335$ & $0.000207585$ & $66.48$\\
\hline
$495$ & $500$ & $1.01834$ & $1.00975$ & $1.79853$ & $0.930467$ & $1.0486$ & $277.44$ & $0.000151062$ & $66.51$\\
\hline
$500$ & $505$ & $1.01867$ & $1.01099$ & $1.82223$ & $0.892651$ & $1.05582$ & $277.559$ & $0.000276115$ & $66.55$\\
\hline
$505$ & $510$ & $1.0187$ & $1.00994$ & $1.80285$ & $0.899901$ & $1.04805$ & $309.723$ & $0.000206119$ & $66.57$\\
\hline
$510$ & $515$ & $1.01851$ & $1.00857$ & $1.76958$ & $0.897381$ & $1.05071$ & $336.08$ & $0.000216335$ & $66.59$\\
\hline
$515$ & $520$ & $1.01841$ & $1.00972$ & $1.74506$ & $0.867681$ & $1.04449$ & $358.983$ & $0.000129439$ & $66.60$\\
\hline
$520$ & $525$ & $1.01808$ & $1.00909$ & $1.69862$ & $0.91292$ & $1.04574$ & $395.643$ & $0.00014666$ & $66.62$\\
\hline
$525$ & $530$ & $1.01775$ & $1.00811$ & $1.83159$ & $0.881166$ & $1.05596$ & $472.358$ & $0.000228591$ & $66.65$\\
\hline
$530$ & $535$ & $1.01775$ & $1.00761$ & $1.78934$ & $0.899516$ & $1.05953$ & $462.286$ & $0.000296956$ & $66.67$\\
\hline
$535$ & $540$ & $1.01804$ & $1.00799$ & $1.75446$ & $0.864478$ & $1.05182$ & $531.88$ & $0.000272301$ & $66.68$\\
\hline
$540$ & $545$ & $1.01797$ & $1.00778$ & $1.71041$ & $0.969014$ & $1.05986$ & $533.438$ & $0.000323617$ & $66.70$\\
\hline
$545$ & $550$ & $1.01814$ & $1.00869$ & $1.71209$ & $0.867872$ & $1.05977$ & $612.157$ & $0.000172334$ & $66.68$\\
\hline
$550$ & $555$ & $1.01798$ & $1.0077$ & $1.79251$ & $0.931142$ & $1.05695$ & $700.06$ & $0.000188567$ & $66.69$\\
\hline
$555$ & $560$ & $1.01809$ & $1.00941$ & $1.69138$ & $0.919723$ & $1.06404$ & $731.948$ & $0.000180617$ & $66.69$\\
\hline
$560$ & $565$ & $1.01797$ & $1.00911$ & $1.68662$ & $0.864885$ & $1.06303$ & $884.433$ & $0.000261075$ & $66.70$\\
\hline
$565$ & $570$ & $1.01805$ & $1.00675$ & $1.69631$ & $0.874895$ & $1.07243$ & $977.767$ & $0.000189268$ & $66.68$\\
\hline
$570$ & $575$ & $1.01791$ & $1.00703$ & $1.61088$ & $0.903648$ & $1.06953$ & $1008.6$ & $0.000260523$ & $66.69$\\
\hline
$575$ & $580$ & $1.01717$ & $1.00681$ & $1.71789$ & $0.84638$ & $1.07915$ & $1090.48$ & $0.000214881$ & $66.68$\\
\hline
$580$ & $585$ & $1.01789$ & $1.00627$ & $1.75529$ & $0.862506$ & $1.07199$ & $1161.28$ & $0.000137629$ & $66.65$\\
\hline
$585$ & $590$ & $1.01769$ & $1.00851$ & $1.72491$ & $0.81516$ & $1.07599$ & $1243.56$ & $0.000237294$ & $66.66$\\
\hline
$590$ & $595$ & $1.01737$ & $1.00637$ & $1.67399$ & $0.87291$ & $1.07789$ & $1496.72$ & $0.000244104$ & $66.64$\\
\hline
$595$ & $600$ & $1.01778$ & $1.00634$ & $1.6905$ & $0.892796$ & $1.07244$ & $1480.77$ & $0.000267985$ & $66.63$\\
\hline
$600$ & $605$ & $1.01737$ & $1.00657$ & $1.67695$ & $0.909403$ & $1.08182$ & $1757.68$ & $0.000242804$ & $66.61$\\
\hline
$605$ & $610$ & $1.01719$ & $1.00573$ & $1.74453$ & $0.884353$ & $1.08245$ & $1852.2$ & $0.000313867$ & $66.60$\\
\hline
$610$ & $615$ & $1.01787$ & $1.00628$ & $1.67817$ & $0.924263$ & $1.09403$ & $2201.2$ & $0.000225138$ & $66.57$\\
\hline
$615$ & $620$ & $1.01779$ & $1.00529$ & $1.71989$ & $0.89541$ & $1.10734$ & $2358.51$ & $0.000269884$ & $66.56$\\
\hline
$620$ & $625$ & $1.01766$ & $1.00591$ & $1.68201$ & $0.898451$ & $1.09211$ & $2557.47$ & $0.00023344$ & $66.53$\\
\hline
$625$ & $630$ & $1.01739$ & $1.00679$ & $1.63793$ & $0.953277$ & $1.08654$ & $2762.17$ & $0.000242996$ & $66.51$\\
\hline
$630$ & $635$ & $1.01721$ & $1.00683$ & $1.65337$ & $0.870998$ & $1.0922$ & $2875.57$ & $0.000338419$ & $66.49$\\
\hline
$635$ & $640$ & $1.01726$ & $1.00489$ & $1.69949$ & $0.844695$ & $1.09723$ & $3334.1$ & $0.000222066$ & $66.45$\\
\hline
$640$ & $645$ & $1.0169$ & $1.0052$ & $1.72567$ & $0.843783$ & $1.08135$ & $3258.7$ & $0.000165915$ & $66.41$\\
\hline
$645$ & $650$ & $1.01764$ & $1.00539$ & $1.70071$ & $0.916906$ & $1.09723$ & $3603.31$ & $0.000177251$ & $66.38$\\
\hline
$650$ & $655$ & $1.01739$ & $1.0043$ & $1.56953$ & $0.951441$ & $1.09117$ & $4037.06$ & $0.000149961$ & $66.35$\\
\hline
$655$ & $660$ & $1.01748$ & $1.00401$ & $1.63269$ & $0.886415$ & $1.09295$ & $4648.06$ & $0.000186047$ & $66.32$\\
\hline
$660$ & $665$ & $1.01649$ & $1.0049$ & $1.72418$ & $0.886985$ & $1.09507$ & $4673.86$ & $0.00017414$ & $66.29$\\
\hline
$665$ & $670$ & $1.01674$ & $1.00599$ & $1.67976$ & $0.869346$ & $1.10091$ & $5572.25$ & $0.000177106$ & $66.25$\\
\hline
$670$ & $675$ & $1.01663$ & $1.0046$ & $1.61937$ & $0.83382$ & $1.1006$ & $5675.14$ & $0.000245665$ & $66.22$\\
\hline
$675$ & $680$ & $1.01694$ & $1.00428$ & $1.65215$ & $0.872436$ & $1.11171$ & $6423.04$ & $0.00011767$ & $66.16$\\
\hline
$680$ & $685$ & $1.01709$ & $1.00442$ & $1.66715$ & $0.866925$ & $1.11875$ & $7423.57$ & $0.000142691$ & $66.13$\\
\hline
$685$ & $690$ & $1.01638$ & $1.00529$ & $1.58343$ & $0.863125$ & $1.11827$ & $8229.43$ & $0.000242883$ & $66.11$\\
\hline
$690$ & $695$ & $1.01616$ & $1.0052$ & $1.63037$ & $0.85535$ & $1.1079$ & $8876.13$ & $0.00018326$ & $66.06$\\
\hline
$695$ & $700$ & $1.01678$ & $1.00377$ & $1.60685$ & $0.847676$ & $1.1222$ & $9005.87$ & $0.000360569$ & $66.04$\\
\hline
$700$ & $705$ & $1.01663$ & $1.00434$ & $1.6478$ & $0.832542$ & $1.10678$ & $11011.9$ & $0.000259432$ & $65.98$\\
\hline
$705$ & $710$ & $1.01626$ & $1.00427$ & $1.6748$ & $0.858014$ & $1.13075$ & $11374.6$ & $0.000161921$ & $65.92$\\
\hline
$710$ & $715$ & $1.01702$ & $1.00388$ & $1.66823$ & $0.904501$ & $1.12282$ & $12579.1$ & $0.000253355$ & $65.89$\\
\hline
$715$ & $720$ & $1.01654$ & $1.0052$ & $1.57025$ & $0.91045$ & $1.12799$ & $13030.7$ & $0.000164892$ & $65.83$\\
\hline
$720$ & $725$ & $1.01554$ & $1.00395$ & $1.61751$ & $0.875735$ & $1.12081$ & $14543.9$ & $0.000163032$ & $65.79$\\
\hline
$725$ & $730$ & $1.01616$ & $1.0048$ & $1.61618$ & $0.894992$ & $1.13861$ & $16398.8$ & $0.000164328$ & $65.74$\\
\hline
$730$ & $735$ & $1.01616$ & $1.0052$ & $1.58148$ & $0.881893$ & $1.13236$ & $18388.7$ & $0.000250531$ & $65.70$\\
\hline
$735$ & $740$ & $1.01633$ & $1.00404$ & $1.58224$ & $0.884335$ & $1.12985$ & $21073.1$ & $0.000240423$ & $65.65$\\
\hline
$740$ & $745$ & $1.0158$ & $1.00347$ & $1.62662$ & $0.823366$ & $1.12091$ & $22637.4$ & $0.00019436$ & $65.59$\\
\hline
$745$ & $750$ & $1.01726$ & $1.00319$ & $1.55616$ & $0.837796$ & $1.13674$ & $22591$ & $0.000274638$ & $65.56$\\
\hline
$750$ & $755$ & $1.01604$ & $1.00424$ & $1.66054$ & $0.886899$ & $1.13739$ & $26085.9$ & $0.000176809$ & $65.49$\\
\hline
$755$ & $760$ & $1.01606$ & $1.00397$ & $1.63481$ & $0.866679$ & $1.12632$ & $29306.6$ & $0.000171531$ & $65.44$\\
\hline
$760$ & $765$ & $1.0165$ & $1.00346$ & $1.67097$ & $0.847691$ & $1.13233$ & $29419$ & $0.00019725$ & $65.39$\\
\hline
$765$ & $770$ & $1.01591$ & $1.00381$ & $1.47213$ & $0.900812$ & $1.14082$ & $34118.8$ & $0.000185818$ & $65.33$\\
\hline
$770$ & $775$ & $1.01603$ & $1.00414$ & $1.59268$ & $0.928957$ & $1.15879$ & $38353.4$ & $0.000156928$ & $65.27$\\
\hline
$775$ & $780$ & $1.01618$ & $1.00278$ & $1.49228$ & $0.8967$ & $1.14947$ & $38574.2$ & $0.000226631$ & $65.23$\\
\hline
$780$ & $785$ & $1.01611$ & $1.00268$ & $1.56406$ & $0.877608$ & $1.14836$ & $47867.9$ & $0.000207654$ & $65.17$\\
\hline
$785$ & $790$ & $1.0164$ & $1.00312$ & $1.51167$ & $0.858272$ & $1.16038$ & $50635.3$ & $0.00028385$ & $65.12$\\
\hline
$790$ & $795$ & $1.01537$ & $1.00393$ & $1.64687$ & $0.861995$ & $1.15151$ & $55750.3$ & $0.000200208$ & $65.06$\\
\hline
$795$ & $800$ & $1.0153$ & $1.0029$ & $1.47289$ & $0.887874$ & $1.16342$ & $60283.2$ & $0.000154351$ & $64.99$\\
\hline
$800$ & $805$ & $1.01527$ & $1.00401$ & $1.55174$ & $0.865362$ & $1.1537$ & $66111.4$ & $0.000195508$ & $64.94$\\
\hline
$805$ & $810$ & $1.01574$ & $1.0029$ & $1.5166$ & $0.902223$ & $1.17073$ & $74262.8$ & $0.000225813$ & $64.89$\\
\hline
$810$ & $815$ & $1.01522$ & $1.00339$ & $1.60073$ & $0.876369$ & $1.15792$ & $77118.3$ & $0.000266209$ & $64.83$\\
\hline
$815$ & $820$ & $1.01601$ & $1.00238$ & $1.66365$ & $0.906645$ & $1.15121$ & $84767.1$ & $0.000161823$ & $64.76$\\
\hline
$820$ & $825$ & $1.01583$ & $1.00271$ & $1.6235$ & $0.928793$ & $1.16244$ & $107264$ & $0.000234246$ & $64.71$\\
\hline
$825$ & $830$ & $1.01563$ & $1.00315$ & $1.5685$ & $0.859089$ & $1.15991$ & $106225$ & $0.000187707$ & $64.64$\\
\hline
$830$ & $835$ & $1.01536$ & $1.00383$ & $1.53532$ & $0.868098$ & $1.16856$ & $112590$ & $0.000132111$ & $64.57$\\
\hline
$835$ & $840$ & $1.01556$ & $1.0038$ & $1.56455$ & $0.926623$ & $1.17023$ & $134713$ & $0.000235556$ & $64.53$\\
\hline
$840$ & $845$ & $1.01502$ & $1.00246$ & $1.51063$ & $0.874939$ & $1.1639$ & $132092$ & $0.000207851$ & $64.47$\\
\hline
$845$ & $850$ & $1.01563$ & $1.00206$ & $1.53822$ & $0.86811$ & $1.17549$ & $152832$ & $0.000178814$ & $64.40$\\
\hline
$850$ & $855$ & $1.01498$ & $1.00283$ & $1.63544$ & $0.845347$ & $1.16767$ & $148866$ & $0.000158165$ & $64.34$\\
\hline
$855$ & $860$ & $1.0154$ & $1.00294$ & $1.52587$ & $0.895686$ & $1.16311$ & $186060$ & $0.000141116$ & $64.27$\\
\hline
$860$ & $865$ & $1.01506$ & $1.00221$ & $1.59709$ & $0.850172$ & $1.1766$ & $192984$ & $0.000238358$ & $64.22$\\
\hline
$865$ & $870$ & $1.01519$ & $1.00193$ & $1.54244$ & $0.88004$ & $1.17385$ & $222438$ & $0.000186894$ & $64.15$\\
\hline
$870$ & $875$ & $1.01542$ & $1.00225$ & $1.52497$ & $0.851648$ & $1.18563$ & $242664$ & $0.00028923$ & $64.10$\\
\hline
\end{longtable}
\end{center}
\clearpage
\printbibliography

\end{document}